\documentclass[12pt]{amsart}
\usepackage[margin=1.7cm]{geometry}

\RequirePackage{amsthm,amscd,comment,mathtools}

\usepackage{kantlipsum}
\usepackage[all]{xy}
\usepackage{mathabx}
\usepackage{mathrsfs}
\usepackage{amsmath}
\usepackage{amsfonts}
\usepackage{amssymb}
\usepackage{graphicx}
\usepackage{tikz-cd}
\usepackage[colorinlistoftodos]{todonotes}
\usepackage[pagebackref,breaklinks,colorlinks,linkcolor=red,anchorcolor=red,citecolor=blue]{hyperref}

\usepackage{enumitem}

\usepackage{verbatim}
\usepackage{rotating}

\calclayout

\theoremstyle{definition}

\newtheorem{definition}{Definition}[section]
\newtheorem{remark}[definition]{Remark}

\newtheorem{question}{Question}

\theoremstyle{plain}

\newtheorem{theorem}[definition]{Theorem}
\newtheorem{proposition}[definition]{Proposition}

\newtheorem{lemma}[definition]{Lemma}

\newtheorem{convention}[definition]{Convention}

\newtheorem{example}[definition]{Example}
\numberwithin{equation}{subsection}

\newcommand{\id}{\mathrm{id}}

\newcommand{\Hom}{\operatorname{Hom}}
\newcommand{\HHom}{{\underline{\operatorname{Hom}}}}
\newcommand{\Ext}{\operatorname{Ext}}
\newcommand{\End}{\operatorname{End}}
\newcommand{\EEnd}{{\underline{\operatorname{End}}}}

\newcommand{\Spec}{\operatorname{Spec}}
\newcommand{\Spf}{\operatorname{Spf}}
\newcommand{\Spa}{\operatorname{Spa}}
\renewcommand{\ker}{\operatorname{ker}}

\newcommand{\im}{\operatorname{im}}

\newcommand{\tr}{\operatorname{tr}}
\newcommand{\rk}{\operatorname{rk}}
\newcommand{\gr}{\operatorname{gr}}

\newcommand{\frakm}{{\mathfrak m}}

\newcommand{\frakS}{{\mathfrak S}}

\newcommand{\frakU}{{\mathfrak U}}

\newcommand{\frakX}{{\mathfrak X}}

\newcommand{\bC}{{\mathbb C}}
\newcommand{\bD}{{\mathbb D}}

\newcommand{\bL}{{\mathbb L}}

\newcommand{\bQ}{{\mathbb Q}}

\newcommand{\bZ}{{\mathbb Z}}

\newcommand{\bfC}{{\mathbf C}}

\newcommand{\calE}{{\mathcal E}}
\newcommand{\calF}{{\mathcal F}}

\newcommand{\calH}{{\mathcal H}}
\newcommand{\calI}{{\mathcal I}}

\newcommand{\calL}{{\mathcal L}}

\newcommand{\calO}{{\mathcal O}}

\newcommand{\calQ}{{\mathcal Q}}

\newcommand{\calV}{{\mathcal V}}

\newcommand{\calX}{{\mathcal X}}
\newcommand{\calY}{{\mathcal Y}}

\newcommand{\rH}{{\mathrm H}}

\newcommand{\Qp}{{\bQ_p}}

\newcommand{\Cp}{{\bC_p}}

\newcommand{\Exp}{{\mathrm{Exp}}}           

\newcommand{\HIG}{{\mathrm{HIG}}}           

\newcommand{\Vect}{{\mathrm{Vect}}}         

\newcommand{\et}{{\mathrm{\acute{e}t}}}    

\newcommand{\proet}{{\mathrm{pro\acute{e}t}}}
\newcommand{\tor}{{\mathrm{tor}}}           

\makeatletter

\begin{document}

\title{Lift-independence problem in the $p$-adic Simpson correspondence for curves}

\author{Xiangyu Pan}
\address{School of Mathematical Sciences, Peking University, Beijing, China}
\email{XiangyuPan@stu.pku.edu.cn}

\author{Jiahong Yu}
\address{Academy of Mathematics and Systems Science, Chinese Academy of Sciences, Beijing, China;  
Morningside Center of Mathematics, Chinese Academy of Sciences, Beijing, China}
\email{\nolinkurl{yu_jh@amss.ac.cn}}

\date{\today}

\begin{abstract}
Let $X$ be a proper smooth rigid analytic variety over a complete
algebraically closed $p$-adic field $\mathbf C$. Fix a continuation
$\operatorname{Exp}$ of $\exp$. Faltings, in the curve case, and Heuer
showed that any lifting $\widetilde X$ of $X$ over
$B_{\mathrm{dR}}^+/t^2$ induces an equivalence between the category of
Higgs bundles on $X_{\text{\'et}}$ and the category of $v$-vector bundles
on $X_v$. Changing $\widetilde X$ is known to alter the correspondence
nontrivially. This raises the natural question of how the isomorphism class
of the Higgs bundle associated with a fixed $v$-vector bundle varies with
$\widetilde X$, a question that has not been systematically studied.

In this paper, we address this question when $X$ is a curve of genus
$g\ge 2$. More precisely,
we call a Higgs bundle lift-independent if it corresponds to the same
$v$-bundle under the $p$-adic Simpson correspondence for every lifting
$\widetilde X$. We prove the following results.
\begin{itemize}
\item Every lift-independent Higgs bundle of rank $r\le g$ is nilpotent.
\item Every semistable lift-independent Higgs bundle of
rank $r\le \sqrt{g}+1$ has zero Higgs field.
\item In contrast, vanishing fails even within the class singled out by
Faltings' conjecture: there always exists a lift-independent semistable Higgs
bundle of degree $0$ with nonzero Higgs field.
\end{itemize}
\end{abstract}

\keywords{$p$-adic Simpson correspondence; lift-independence;
Higgs bundles; $v$-vector bundles; $p$-adic non-abelian Hodge theory}
\subjclass[2020]{14G22, 14F30, 14H60}
\maketitle
\tableofcontents

\section{Introduction}

Fix a complete algebraically closed field ${\mathbf{C}}$ over $\mathbb{Q}_p$ with the ring of integers $\mathcal{O}_{\bfC}$. Put $A_{\text{inf}} = \mathbb{A}_{\text{inf}}(\mathcal{O}_{\mathbf{C}})$ and $B_{\text{dR}}^+ = \mathbb{B}_{\text{dR}}^+({\mathbf{C}})$. Fix a generator $\xi$ of $\ker(\theta: A_{\text{inf}} \to {\mathcal{O}_\mathbf{C}})$ and put 
\[A_2:=A_{\text{inf}}/\xi^2 \text{ and } B_2:=B_{\text{dR}}^+/\xi^2.\]
Denote by $\mathrm{Exp}:\bfC\to 1+\frakm_{\bfC}$ a section of the logarithmic map $\log:1+\frakm_{\bfC}\to\bfC$ extending the usual $p$-adic exponential map $\exp$. Then we have the following $p$-adic Simpson correspondence.
\begin{theorem}[Faltings \cite{Faltings_2005}, Heuer \cite{heuer2023padic}]\label{theorem:Faltings-Heuer}
	For a proper smooth rigid analytic variety $X$ over $\bfC$, any flat lifting $\widetilde X$ of $X$ over $B_2$ together with a choice of $\Exp$ induces an equivalence
	\[S_{\Exp,\widetilde X}:\HIG(X_\et)\to \Vect(X_v)\]
	from the category of Higgs bundles on $X_\et$ to the category of $v$-bundles on $X_v$ that is compatible with ranks, tensor products, duals and cohomologies.
\end{theorem}

	We emphasize that the equivalence $S_{\Exp,\widetilde X}$ depends on both $\Exp$ and $\widetilde X$, so as long as we fix an $\Exp$, it only depends on the choice of lifting $\widetilde X$ of $X$. A natural question is how the equivalence $\HIG(X_\et)\to \Vect(X_v)$ depends on such a lifting $\widetilde X$. From now on, we fix an $\Exp$ and denote by
	\[S_{\widetilde X}:=S_{\Exp,\widetilde X}:\HIG(X_\et)\xrightarrow{\simeq} \Vect(X_v)\]
	the above equivalence of categories.
	
	Note that for any two liftings $\widetilde X_1$ and $\widetilde X_2$ of $X$ and any Higgs bundle $(\calH,\theta)$ on $X_\et$, $S_{\widetilde X_1}^{-1}(S_{\widetilde X_2}(\calH,\theta))$ is another Higgs bundle on $X_\et$. A Higgs bundle $(\calH,\theta)\in \HIG(X_\et)$ is called \emph{lift-independent} if for any liftings $\widetilde X_1$ and $\widetilde X_2$, there always exists an isomorphism
	\[(\calH,\theta)\cong S_{\widetilde X_1}^{-1}(S_{\widetilde X_2}(\calH,\theta)).\]
	The following question is natural: What does a lift-independent Higgs bundle look like?
	
	A trivial observation is that any Higgs bundle $(\calH,\theta)$ with trivial Higgs field (i.e., $\theta = 0$) is lift-independent. So one may ask if this is the only case for lift-independent Higgs bundles. That is, we have the following question:
	\begin{question}\label{ques: indep higgs to 0?}
		Let $(\calH,\theta)$ be a lift-independent Higgs bundle on $X_\et$. Is it always true that $\theta = 0$?
	\end{question}
	The paper is devoted to answering this question when $X$ is a proper smooth curve and on the integral level.
	
\subsection{Main results}	
  Suppose that $X$ is a smooth rigid analytic variety with a smooth formal model $\frakX$ over $\calO_{\bfC}$. Then we have the following result on integral $p$-adic Simpson correspondence.
  \begin{theorem}[\cite{abbes2016padic,MinWang2024,anschutz2023smallpadicsimpsoncorrespondence,sheng2024smallpadic}]
  	Any flat lifting $\widetilde \frakX$ of $\frakX$ over $A_2$ induces an equivalence 
  	\[\frakS_{\widetilde \frakX}:\Vect^{\rH\text{-}{\rm small}}(X_v,\widehat \calO_X^+)\to \HIG^{\rH\text{-}{\rm small}}(\frakX_{\et})\]
  	from the category of Hitchin-small $\widehat \calO_X^+$-bundles on $X_v$ to the category of Hitchin-small Higgs bundles on $\frakX_\et$ that is compatible with ranks, tensor products and duals.
  \end{theorem}
  Similarly, when $\frakX$ admits a flat $A_{2}$-lifting, a $\frac{1}{p-1}$-Hitchin-small Higgs bundle $(\calH,\theta)$ on $\frakX_\et$ is called \emph{integrally lift-independent} if for any two liftings $\widetilde \frakX_1$ and $\widetilde \frakX_2$ of $\frakX$, there exists an isomorphism
  \[(\calH,\theta)\cong \frakS_{\widetilde \frakX_1}(\frakS_{\widetilde \frakX_2}^{-1}(\calH,\theta)).\]
  So Question \ref{ques: indep higgs to 0?} also makes sense on the integral level: 
  \begin{question}\label{ques: indep higgs to 0? integral}
  	 Does an integrally lift-independent Higgs bundle always have trivial Higgs field? 
  \end{question}
  We will try to attack Question \ref{ques: indep higgs to 0? integral} when $\frakX$ is a smooth curve in this paper. Here is our main result.
 
 \begin{theorem}[Theorems \ref{theo:rank_le_genus_nilpotent} and \ref{theo: rk leq sqrt g-1 ss+int independ to 0}]\label{theo:integral_rigidity}
 	Let $\mathfrak{X}$ be a projective smooth formal curve over $\mathcal{O}_{\mathbf{C}}$ of genus $g \geqslant 2$ with generic fiber $X$. Let $(\mathcal{H}, \theta)$ be a $\frac{1}{p-1}$-Hitchin-small Higgs bundle on $\mathfrak{X}$ of rank $r\geqslant 1$.
    \begin{enumerate}
        \item If $r\leqslant g$ and $(\calH,\theta)$ is integrally lift-independent, then $\theta^r=0$; in particular, $(\calH,\theta)$ is nilpotent.
        \item If the pull-back $(\calH[\frac1p],\theta[\frac{1}{p}])$ to $X_\et$ is semistable and $r\leqslant\sqrt{g}+1$, then $(\calH,\theta)$ is integrally lift-independent if and only if $\theta=0$.
    \end{enumerate}
 \end{theorem}

\begin{remark}\label{rmk:moduli-perspective}
    Motivated by the moduli-theoretic approach to the $p$-adic Simpson correspondence developed by Heuer and Xu \cite{HeuerXu2026moduli}, an ongoing project of ours aims to eliminate the dependence on the formal model $\mathfrak{X}$, at least for curves, by tracking the lift-independence directly within rigid analytic moduli spaces. In parallel, exploring lifting obstructions at the integral level via the Simpson gerbe \cite[Constructions 12.2.3 and 12.4.4]{BhattpHTnotes} constitutes another active direction of our research.
\end{remark}

 Although Question \ref{ques: indep higgs to 0? integral} has a positive answer in some cases, it is not true in general even if we consider semistable Higgs bundles of degree $0$. 
 
 \begin{theorem}[ = Theorems \ref{theo: nonzero-nilp-deg0-ss-rat-indep-Higgs} and \ref{thm:integral-nonnilpotent-example}]\label{theorem:counterexample}
    \begin{enumerate}
    	\item Let $X$ be a projective smooth rigid curve over $\mathbf C$ with genus $g\geqslant 2$. There always exists a nilpotent semistable Higgs bundle of degree $0$ with nonzero Higgs field that is lift-independent.
    	
    	\item For all sufficiently large primes $p$, there exist a complete algebraically closed extension $C/\mathbb{Q}_p$, a projective smooth formal curve $\frakX$ over $\calO_C$ of genus $2$, and an integrally lift-independent non-nilpotent $\frac{1}{p-1}$-Hitchin-small Higgs bundle $(\calH,\theta)$ on $\frakX_\et$ whose generic fiber is semistable of degree $0$.
    \end{enumerate}
 \end{theorem}
 
 We explain why semistability is needed for the vanishing assertion in Theorem \ref{theo:integral_rigidity}(2).
 
  \begin{example}[ = Example \ref{eg: without semi stable}]\label{exam:counterexample}
 		Suppose that $X$ is a projective smooth curve over $\mathbf C$ of genus $g\geqslant 2$. Let $\mathcal{L}$ be a line bundle of degree $d > 2g - 2$. For any nonzero global section $s \in  \mathrm{H}^0(X, \mathcal{L} \otimes \Omega_X(-1))$, we define a Higgs bundle $(\calH_s,\theta_s)$ of rank $2$ such that
 		\[\calH_s = \mathcal{L} \oplus \mathcal{O}_X\text{ and }\theta_s = \begin{pmatrix} 0 & s \\ 0 & 0 \end{pmatrix} : \mathcal{H}_{s} \xrightarrow{\mathrm{projection}} \mathcal{O}_X \xrightarrow{s} \mathcal{L} \otimes \Omega_{X}(-1)  \xrightarrow{\subset} \mathcal{H}_s \otimes \Omega_{X}(-1) .\]
 		Then $(\mathcal{H}_s, \theta_s)$ is not semistable but lift-independent.
 	\end{example}

 \begin{remark}
 	For a projective smooth curve $X$ over $\mathbf C = \Cp$, the complete algebraic closure of $\Qp$, Faltings conjectured that any semistable Higgs bundle of degree $0$ should correspond to a $\mathbf C$-local system on $X_\et$ via the $p$-adic Simpson correspondence described in Theorem \ref{theorem:Faltings-Heuer}. Accepting Faltings' conjecture, we hope that, if a $\mathbf{C}$-local system over $X$ is lift-independent, then its $p$-adic Simpson correspondence is compatible with the complex Simpson correspondence.
 \end{remark}
 
 Now, we explain the key ingredient to obtain Theorems \ref{theo:integral_rigidity} and \ref{theorem:counterexample}. Note that for any Higgs bundle $(\calH,\theta)$ on $X_\et$, the Higgs field induces a natural morphism 
 \begin{equation}\label{equ:H(theta)-map}
 	\rH^1(\theta):\rH^1(X_\et,T_X(1))\to \rH^1(X_\et,\EEnd_{\calO_X}(\calH,\theta)).
 \end{equation}
  If the natural morphism \eqref{equ:H(theta)-map} is zero, then we say $(\calH,\theta)$ is \emph{cohomologically lift-independent}. The key observation is the following theorem:
  
 \begin{theorem}[ = Theorem \ref{theo: int lift independ to quasi independ}]\label{theo: intro int lift independ to coho independ}
 	Let $\mathfrak{X}$ be a proper smooth formal scheme over $\mathcal{O}_{\mathbf{C}}$ with generic fiber $X$, admitting a flat $A_2$-lifting. Let $(\calH,\theta)$ be a $\frac{1}{p-1}$-Hitchin-small Higgs bundle on $\frakX_\et$ with the pull-back $(\calH[\frac1p],\theta[\frac{1}{p}])$ on $X_\et$. If $(\calH,\theta)$ is integrally lift-independent, then $(\calH[\frac1p],\theta[\frac{1}{p}])$ is cohomologically lift-independent.
 \end{theorem}

\begin{remark}
A closely related cohomological condition appears in the work of
Hu--Sun--Zuo \cite{HuSunZuo2025Isomonodromic}.
Let $(E,\theta)$ be a Higgs bundle on a smooth projective complex
curve $X$ of genus $g\geq 2$, and put
\[
  Z_\theta=\EEnd(E,\theta),
  \qquad
  \mathcal C_\theta^\bullet=
  \left[
    \EEnd(E)
    \xrightarrow{[\theta,-]}
    \EEnd(E)\otimes K_X
  \right],
\]
where the complex is concentrated in degrees $0$ and $1$.
Their map factors as
\[
  \theta_* \colon
  H^1(X,T_X)
  \xrightarrow{H^1(\theta)}
  H^1(X,Z_\theta)
  \hookrightarrow
  \mathbb H^1(X,\mathcal C_\theta^\bullet).
\]
Indeed, writing
$Q=\EEnd(E)/Z_\theta$, the short exact sequence
of complexes
\[
  0\longrightarrow Z_\theta[0]
  \longrightarrow \mathcal C_\theta^\bullet
  \longrightarrow
  \left[Q\longrightarrow\EEnd(E)\otimes K_X\right]
  \longrightarrow 0
\]
gives the asserted injection, since the last complex has
vanishing zeroth hypercohomology.

Consequently, $\theta_*=0$ is equivalent to cohomological
lift-independence. In their polystable degree-zero setting,
this condition expresses the vanishing, at the given point,
of the anti-holomorphic derivative of the isomonodromic
deformation in every deformation direction of $X$.
Their low-rank results therefore concern the same cohomological
condition. Theorem \ref{theo: intro int lift independ to coho independ} relates
this condition to integral lift-independence in the $p$-adic
Simpson correspondence, while our vanishing results also
address semistable Higgs bundles without assuming polystability.
\end{remark}

  According to this theorem, to prove Theorem \ref{theo:integral_rigidity}, it suffices to prove the corresponding nilpotence and vanishing results for cohomologically lift-independent Higgs bundles, namely Theorems \ref{theo:rank_le_genus_nilpotent} and \ref{theo: r<(g-1)^{1/2} quasi-independ to 0}. To obtain Example \ref{exam:counterexample} and Theorem \ref{theorem:counterexample}, we first construct cohomologically lift-independent Higgs bundles, and then explain that our examples already satisfy the conditions for being lift-independent.

\subsection{Structure of the paper and use of AI tools}

The paper is organized as follows. In Section \ref{section: Quasi-lift ind}, we recall the integral small $p$-adic Simpson
correspondence for smooth formal schemes and introduce the notions of integral lift-independence
and cohomological lift-independence. We then compare these two notions and prove that integral
lift-independence imposes a cohomological vanishing condition on the generic fiber.

In Section \ref{section: indep to 0}, we study cohomologically lift-independent Higgs bundles on smooth projective curves.
The section proves that, under suitable semistability, rank, and genus hypotheses, such Higgs bundles
must have zero Higgs field. Combining these results with the comparison theorem from Section \ref{section: Quasi-lift ind}
gives the main positive results on integral lift-independence.

In Section \ref{section: non zero independ}, we construct examples showing that nonzero lift-independent Higgs fields do exist
outside the range of the positive results. We first give nilpotent examples with nonzero Higgs fields,
and then construct non-nilpotent examples in degree $0$. These examples show that the
lift-independence problem is sensitive to both semistability and the numerical hypotheses appearing
in the main theorems.

We also disclose the use of AI assistance. During the preparation of this work, the authors used
Rethlas \cite{JuEtAl2026Rethlas} as an auxiliary tool for mathematical exploration, proof search, and exposition.
All AI-assisted suggestions were checked, revised, and incorporated only after independent
verification by the authors. The authors take full responsibility for all statements, proofs, and any
remaining errors in the paper.

\subsection*{Acknowledgments}

The first author would like to express his sincere gratitude to his advisors Ruochuan Liu and Enlin Yang. 
The authors would like to express their sincere gratitude to Yupeng Wang for introducing this problem to us and for his continuous encouragement. 
We are also deeply grateful to Qixiang Wang, Daxin Xu, Qizheng Yin and Mingjia Zhang for their insightful discussions and valuable comments. 
Part of the work was carried out when the authors were visiting the Shanghai Center for Mathematical Sciences at Fudan University. 
Xiangyu Pan was partially supported by the National Key R\&D Program of China (Grant
No.~2021YFA1001400). 

The authors would like to thank the Rethlas team, namely Haocheng Ju, Jiedong
Jiang, Shurui Liu, Guoxiong Gao, Yuefeng Wang, Zeming Sun, Leheng Chen, Bin Wu,
Liang Xiao, and Bin Dong, for their contributions to the development of Rethlas \cite{JuEtAl2026Rethlas}.

\addtocontents{toc}{\protect\setcounter{tocdepth}{-1}}
\section*{Notations and Conventions}

For sheaves $\mathcal F,\mathcal G$ of modules on a ringed space $X$, we write
$\HHom(\mathcal F,\mathcal G)$ for the sheaf of homomorphisms and
$\EEnd(\mathcal F):=\HHom(\mathcal F,\mathcal F)$ for the endomorphism sheaf.
The corresponding groups of global morphisms are denoted without underlines:
\[
    \Hom(\mathcal F,\mathcal G)=\mathrm H^0(X,\HHom(\mathcal F,\mathcal G)),
    \qquad
    \End(\mathcal F)=\mathrm H^0(X,\EEnd(\mathcal F)).
\]
We use the same convention for Higgs bundles. In particular,
$\EEnd(\mathcal H,\theta)$ is the subsheaf of endomorphisms commuting with
$\theta$, and $\End(\mathcal H,\theta)$ is its ring of global sections.

For a complete non-Archimedean field $K$ and its ring of integers $K^+$, by a
\emph{formal scheme} over $K^+$ we always mean a topologically finitely
presented $\pi$-adic formal scheme over $K^+$, where $\pi\in K^+$ is a
pseudo-uniformizer.

For a variety $X$ over $K$ (either algebraic or rigid analytic) and a coherent
$\mathcal O_X$-module $\mathcal F$, when $\mathrm H^i(X,\mathcal F)$ is a
finite-dimensional $K$-vector space (for instance, when $X$ is proper over
$K$), we denote its dimension over $K$ by
\[
    h^i(X,\mathcal F):=\dim_K\mathrm H^i(X,\mathcal F).
\]

For a rigid variety $X/K$, a closed point $s\in X$ with residue field
$\kappa(s)$, and a coherent $\mathcal O_X$-module $\mathcal F$, we write
$\mathcal F(s):=\mathcal F_s\otimes_{\mathcal O_{X,s}}\kappa(s)$ for its fiber,
viewed as a finite-dimensional $\kappa(s)$-vector space. For a morphism
$f:\mathcal F\to\mathcal G$, we write
$f(s):=f_s\otimes1:\mathcal F(s)\to\mathcal G(s)$ for the induced
$\kappa(s)$-linear map. More generally, for a morphism $Y\to S$ of rigid
varieties and a closed point $s\in S$, the same notation denotes pullback of
sheaves or morphisms on $Y$ to the fiber
$Y(s):=Y\times_S\operatorname{Spa}(\kappa(s))$.

Furthermore, let $X$ be an integral variety (either algebraic or rigid
analytic), and let $\mathcal F$ be a coherent $\mathcal O_X$-module. We denote
by $\mathcal F_{\mathrm{tor}}$ the \emph{torsion subsheaf} of $\mathcal F$,
whose local sections are those annihilated by some nonzero local section of
$\mathcal O_X$. We denote by $\mathcal F_{\mathrm{tf}}$ the
\emph{torsion-free quotient} of $\mathcal F$, defined by the canonical short
exact sequence
\[
    0\to\mathcal F_{\mathrm{tor}}\to\mathcal F
    \to\mathcal F_{\mathrm{tf}}\to0.
\]

\addtocontents{toc}{\protect\setcounter{tocdepth}{1}}
 
\section{Lift-Independence and Cohomological Lift-Independence}\label{section: Quasi-lift ind}
The purpose of this section is to relate integral lift-independence to a cohomological vanishing condition. 
We first recall the integral small $p$-adic Simpson correspondence and the period sheaf used in its construction. 
We then compute how the correspondence changes when the $A_2$-lifting varies. 
Finally, this transition formula is used to prove that integral lift-independence implies cohomological lift-independence after inverting $p$. 

\subsection{Recollection of the integral $p$-adic Simpson functor}

In this subsection, we recall the small $p$-adic Simpson correspondence for smooth formal schemes constructed in \cite{MinWang2024}.

Let $\frakX$ be a smooth formal scheme over $\calO _{\mathbf{C}}$ of relative dimension $d$ and with generic fiber $X$. 
Assume that there exists a flat $A_{2}$-lifting $\widetilde{\frakX}$ of $\frakX$. 
Let $\nu : X_{\proet} \to \frakX_{\et}$ be the canonical morphism. 
Then we have a sequence of morphisms of sheaves on the pro-\'etale site $X_{\mathrm{pro\acute{e}t}}$
\[ \nu^{-1}\calO _{\widetilde{\frakX}} \xrightarrow[]{} \nu^{-1}\calO _{\frakX} \xrightarrow[]{} \widehat{\calO }^+_{X},\]
which induces a distinguished triangle of completed cotangent complexes:
\[ \widehat{\bL}_{\calO _{\frakX}/\calO _{\widetilde{\frakX}}} \widehat{\otimes}^{\bL}_{\calO _{\frakX}} \widehat{\calO }^+_{X} \xrightarrow[]{} 
\widehat{\bL}_{\widehat{\calO }^+_{X}/\calO _{\widetilde{\frakX}}} \xrightarrow[]{} 
\widehat{\bL}_{\widehat{\calO }^+_{X}/\calO _{\frakX}} \xrightarrow[]{+1},\]
here we omit the $\nu^{-1}$. 
Taking cohomologies and Tate twists, we obtain the following short exact sequence \cite[Theorem 2.9]{wang2023padic} of $\widehat{\calO}^+_{X}$-modules
\[ 0 \xrightarrow[]{} \widehat{\calO }^+_{X} \xrightarrow[]{} 
\rho \mathrm{H}^{-1}(\widehat{\mathbb{L}}_{\widehat{\calO }^+_{X}/\calO _{\widetilde{\frakX}}})(-1) \xrightarrow[]{} 
\widehat{\calO }^+_{X} \widehat{\otimes}_{\calO _{\frakX}} \rho \widehat{\Omega}^1_{\frakX}(-1) \xrightarrow[]{} 0,\]
where $\rho=\zeta_{p}-1$.

\begin{definition}[{\cite[Theorem 2.9]{wang2023padic}}]\label{def:int-Faltings-ext}
    The sheaf 
    \[\mathcal{E}^+=\mathcal{E}^+_{\widetilde{\frakX}} = \rho \mathrm{H}^{-1}(\widehat{\mathbb{L}}_{\widehat{\calO }^+_{X}/\calO _{\widetilde{\frakX}}})(-1)\]
    on $X_{\mathrm{pro\acute{e}t}}$ is called the integral Faltings extension associated to the $A_{2}$-lifting $\widetilde{\frakX}$.
\end{definition}

\begin{lemma}[{\cite[Lemma 2.1, (2.4)]{MinWang2024}}]\label{lem:Quillen-Illusie}
    There exists an exact sequence 
    \[ 0 \xrightarrow[]{} \Gamma(\widehat{\calO }^+_{X}) \xrightarrow[]{}  \Gamma(\mathcal{E}^+) \xrightarrow[]{\partial} 
    \Gamma(\mathcal{E}^+) \otimes_{\calO _{\frakX}} \rho \widehat{\Omega}^1_{\frakX}(-1) \xrightarrow[]{\partial} \cdots \xrightarrow[]{\partial} 
    \Gamma(\mathcal{E}^+) \otimes_{\calO _{\frakX}} \rho^d\widehat{\Omega}^d_{\frakX}(-d) \xrightarrow[]{} 0.
    \]
    Here $\Gamma(-)$ is the free pd-algebra functor. 
\end{lemma}

Let $e$ be the basis $1$ of $\widehat{\mathcal{O}}_X^+$ as a finite free $\widehat{\mathcal{O}}_X^+$-module. 
Then $\Gamma(\widehat{\mathcal{O}}_X^+) = \widehat{\mathcal{O}}_X^+[e]_{\mathrm{pd}}$ is the free pd-algebra over $\widehat{\mathcal{O}}_X^+$ generated by $e$. 
The element $e - \rho$ admits divided powers of all orders in $\Gamma(\widehat{\mathcal{O}}_X^+)$. 
Let $\mathcal{I}_{\mathrm{pd}}$ be the pd-ideal of $\Gamma(\widehat{\mathcal{O}}_X^+)$ principally generated by $e - \rho$. Then we have $\widehat{\mathcal{O}}_X^+ \cong \Gamma(\widehat{\mathcal{O}}_X^+)/\mathcal{I}_{\mathrm{pd}}$.

\begin{definition}[{\cite[Definition 2.2]{MinWang2024}}]\label{def:period-ring-OCpd}
    \mbox{}
    \begin{enumerate}
        \item Define $\calO \mathbb{C}_{\mathrm{pd}}^{+} = \Gamma(\mathcal{E}^+) \otimes_{\Gamma(\widehat{\calO }^+_{X})} \widehat{\calO }^+_{X}$. 
        Here $\widehat{\calO }^+_{X}$ is regarded as a $\Gamma(\widehat{\calO }^+_{X})$-algebra via the isomorphism $\widehat{\mathcal{O}}_X^+ \cong \Gamma(\widehat{\mathcal{O}}_X^+)/\mathcal{I}_{\mathrm{pd}}$.
        \item Define $\calO \widehat{\mathbb{C}}_{\mathrm{pd}}^{+}=\lim_{n} \calO \mathbb{C}_{\mathrm{pd}}^{+}/p^n$ as the $p$-adic completion of $\calO \mathbb{C}_{\mathrm{pd}}^{+}$.
        \item Denote by $\nabla : \calO \widehat{\mathbb{C}}_{\mathrm{pd}}^{+} \xrightarrow[]{} 
        \calO \widehat{\mathbb{C}}_{\mathrm{pd}}^{+} \widehat{\otimes}_{\calO _{\frakX}} \rho \widehat{\Omega}^{1}_{\frakX}(-1)$ the $\widehat{\calO }^+_{X}$-linear morphism induced by $\partial$ in Lemma \ref{lem:Quillen-Illusie}.
    \end{enumerate}
\end{definition}

The period ring $\calO \widehat{\mathbb{C}}_{\mathrm{pd}}^{+}$ satisfies the Poincar\'e Lemma.

\begin{lemma}[{\cite[Proposition 2.3]{MinWang2024}}]\label{lem:Poincare-lem-for-OCpd}
    The following sequence is exact: 
    \[ 0 \xrightarrow[]{} \widehat{\calO }^+_{X} \xrightarrow[]{} \calO \widehat{\mathbb{C}}_{\mathrm{pd}}^{+} \xrightarrow[]{\nabla} 
    \calO \widehat{\mathbb{C}}_{\mathrm{pd}}^{+} \otimes_{\calO _{\frakX}} \rho\widehat{\Omega}^1_{\frakX}(-1) \xrightarrow[]{\nabla} \cdots \xrightarrow[]{\nabla} 
    \calO \widehat{\mathbb{C}}_{\mathrm{pd}}^{+} \otimes_{\calO _{\frakX}} \rho^d\widehat{\Omega}^d_{\frakX}(-d) \xrightarrow[]{} 0 . \] 
    In particular, $\nabla$ defines a Higgs field on $\calO \widehat{\mathbb{C}}_{\mathrm{pd}}^{+}$.
\end{lemma}

To compare the integral $p$-adic Simpson functor induced by different liftings, we will need an explicit local description of the period sheaf $\calO\widehat{\mathbb C}_{\mathrm{pd}}^+$ on small affine opens.

Following \cite[Convention 2.4]{MinWang2024}, an affine formal scheme $\mathfrak{U} = \mathrm{Spf}(R)$ over $\mathcal{O}_{\mathbf{C}}$ of relative dimension $d$ is called small if there is an \'etale morphism $\square : \calO_{\bfC}\langle T_1^{\pm 1}, \dots, T_d^{\pm 1}\rangle \to R$. 
Such an \'etale morphism $\square$ is called a toric chart on $\mathfrak{U}$. 
In this case, we can deduce from the smoothness of $R$ that, up to isomorphisms, there is a unique $A_2$-lifting $\widetilde{\frakU} = \mathrm{Spf}(\widetilde{R})$ of $\frakU$. 
By the \'etaleness of $\square$, there exists a unique $A_2$-morphism $A_2\langle T_1^{\pm 1}, \dots, T_d^{\pm 1}\rangle \to \widetilde{R}$ lifting $\square$.

Let $U$ be the generic fiber of $\mathfrak{U}$ and let $U_\infty$ be the base-change of $U$ along the morphism
\[
\begin{aligned}
&\mathrm{Spa}(\bfC\langle T_1^{\pm 1/p^\infty}, \dots, T_d^{\pm 1/p^\infty}\rangle, \calO_{\bfC}\langle T_1^{\pm 1/p^\infty}, \dots, T_d^{\pm 1/p^\infty}\rangle) \\
&\longrightarrow \mathrm{Spa}(\bfC\langle T_1^{\pm 1}, \dots, T_d^{\pm 1}\rangle, \calO_{\bfC}\langle T_1^{\pm 1}, \dots, T_d^{\pm 1}\rangle).
\end{aligned}
\]

Then $U_\infty$ is a perfectoid space such that $U_\infty \to U$ is a pro-\'etale Galois cover with Galois group
\[
\Gamma \cong \mathbb{Z}_p\gamma_1 \oplus \cdots \oplus \mathbb{Z}_p\gamma_d,
\]
where for any $1 \leqslant i,j \leqslant d$ and any $n \geqslant 1$, $\gamma_i$ is determined by sending $T_j^{1/p^n}$ to $\zeta_{p^n}^{\delta_{ij}} T_j^{1/p^n}$ and $\delta_{ij}$ denotes Kronecker's $\delta$-function. 

\begin{lemma}[{\cite[Lemma 2.5]{MinWang2024}}]\label{lem:local-description-of-int-Faltings-ext}
    Keep the notations above. There exist sections $\rho y_1, \dots, \rho y_d$ of $\mathcal{E}^+|_{U_{\infty}} $ lifting $\frac{\rho}{t}\mathrm{d}\log T_1, \dots ,\frac{\rho}{t}\mathrm{d}\log T_d$ via the projection
    \[ \calE^+|_{U_{\infty}} \to \widehat{\calO }_{U_{\infty}}^+ \widehat{\otimes}_{\calO _{\frakU }} \rho\widehat{\Omega}^1_{\frakU }(-1)\]
    such that as an $\widehat{\calO }_{U_{\infty}}^+$-module, $\calE^+|_{U_{\infty}} \cong \widehat{\calO }_{U_{\infty}}^+ e \oplus (\bigoplus_{i=1}^d \widehat{\calO }_{U_{\infty}}^+ \rho y_i)$ and for any $1 \leqslant i,j \leqslant d$, $\gamma_i(\rho y_j) = \rho y_j + \rho\delta_{ij}e$.
\end{lemma}

\begin{lemma}[{\cite[Corollary 2.6]{MinWang2024}}]\label{lem:local-description-of-OCpd}
    Keep the notations above. There exists an isomorphism of $\widehat{\mathcal{O}}_U^+$-algebras
    \[ \iota : \widehat{\mathcal{O}}_{U}^{+}[\rho Y_1, \dots, \rho Y_d]^{\wedge}_{\mathrm{pd}}\big|_{U_\infty} \xrightarrow[]{\sim} \mathcal{O}\widehat{\mathbb{C}}_{\mathrm{pd}}^+\big|_{U_\infty} \]
    by identifying $\rho Y_i$'s with the images of $\rho y_i$'s via the composition $\calE^+ \to \Gamma(\calE^+) \to \mathcal{O}\widehat{\mathbb{C}}_{\mathrm{pd}}^+$.
    Via this isomorphism, the Higgs field $\nabla$ on $\mathcal{O}\widehat{\mathbb{C}}_{\mathrm{pd}}^+\big|_{U_\infty}$ is given by 
    \[ \nabla = \sum_{i=1}^d \frac{\partial}{\partial Y_i} \otimes \frac{\mathrm{d}\log T_i}{t} : \mathcal{O}\widehat{\mathbb{C}}_{\mathrm{pd}}^+\big|_{U_\infty} \to \mathcal{O}\widehat{\mathbb{C}}_{\mathrm{pd}}^+\big|_{U_\infty} \widehat{\otimes} \rho\widehat{\Omega}_{\frakU }^1(-1) = \bigoplus_{i=1}^d \mathcal{O}\widehat{\mathbb{C}}_{\mathrm{pd}}^+\big|_{U_\infty} \cdot \rho \frac{\mathrm{d}\log T_i}{t}\]
    by identifying $\widehat{\Omega}_{\frakU }^1(-1)$ with $\bigoplus_{i=1}^d \calO _{\frakU }\cdot \frac{\mathrm{d}\log T_i}{t}$.
\end{lemma}

Recall the definition of Hitchin-small Higgs bundles: 

\begin{definition}[{\cite[Definition 4.2]{MinWang2024} \& \cite[Definition 3.2]{anschutz2023smallpadicsimpsoncorrespondence}}]
    Assume $a \geqslant \frac{1}{p-1}$. 
    By an $a$-Hitchin small Higgs bundle of rank $r$ on $\mathfrak{X}_{\mathrm{\acute{e}t}}$, we mean a pair $(\mathcal{H}, \theta)$ of a locally finite free $\mathcal{O}_{\mathfrak{X}}$-module $\mathcal{H}$ of rank $r$ and an $\mathcal{O}_{\mathfrak{X}}$-linear morphism $\theta : \mathcal{H} \to \mathcal{H} \otimes_{\mathcal{O}_{\mathfrak{X}}} \rho\widehat{\Omega}_{\mathfrak{X}}^1(-1)$ satisfying $\theta \wedge \theta = 0$ such that
    for any $\partial\in \rho^{-1}p^{\frac{1}{p-1}-a}(\widehat{\Omega}^1_{\mathfrak{X}}(-1))^{\vee}$, the action of $\partial$ on $\mathcal{H}$ is topologically nilpotent.
\end{definition}

\begin{remark}
    This definition interpolates between those of Min--Wang (which is essentially Faltings' smallness) and Anschütz--Heuer--Le Bras, with the latter using the Breuil--Kisin twist instead of Tate twist. Hence, they differ from each other by a factor of $(\zeta_p-1)$. 
\end{remark}

Using the period ring $\calO \widehat{\mathbb{C}}_{\mathrm{pd}}^{+}$, 
\cite{MinWang2024} constructed the following integral $p$-adic Simpson correspondence.

\begin{theorem}[{\cite[Theorem 1.1]{MinWang2024}}]\label{thm:main-thm-of-MinWang}
    Assume $a\geqslant  \frac{1}{p-1}$. 
    \begin{enumerate}
        \item For any $a$-Hitchin small $\widehat{\calO }^+_{X}$-representation $\mathcal{L}$ of rank $r$ on $X_{\mathrm{pro\acute{e}t}}$,  
        $(\nu_{*}(\mathcal{L} \otimes_{\widehat{\calO }^+_{X}}\calO \widehat{\mathbb{C}}_{\mathrm{pd}}^{+}), \nu_{*}(1\otimes \nabla))$ 
        is an $a$-Hitchin small Higgs bundle of rank $r$ on $\frakX_{\mathrm{\acute{e}t}}$. 
        \item For any  $a$-Hitchin small Higgs bundle $(\mathcal{H},\theta)$ of rank $r$ on $\frakX_{\mathrm{\acute{e}t}}$, 
        $(\mathcal{H} \otimes_{\calO _{\frakX}}\calO \widehat{\mathbb{C}}_{\mathrm{pd}}^{+})^{\theta\otimes 1+ 1\otimes \nabla=0}$ 
        is an $a$-Hitchin small $\widehat{\calO }^+_{X}$-representation of rank $r$ on $X_{\mathrm{pro\acute{e}t}}$. 
        \item The functors in (1) and (2) are quasi-inverse to each other and hence define an equivalence of categories 
        \[ \mathfrak{S}_{\widetilde{\frakX}} : 
        \{\text{$a$-Hitchin small $\widehat{\mathcal{O}}_X^+$-representations on $X_{\mathrm{pro\acute{e}t}}$} \} 
        \xrightarrow{\sim} \{\text{$a$-Hitchin small Higgs bundles on $\mathfrak{X}$}\}.\]
    \end{enumerate}
\end{theorem}

\subsection{Transition formula and cohomological lift-independence}

Our main goal is to prove Theorem \ref{theo: int lift independ to quasi independ}, which is the precise form of Theorem \ref{theo: intro int lift independ to coho independ} stated in the introduction.

We begin with the key local calculation, which describes how the integral $p$-adic Simpson correspondence changes when the lifting is changed. This calculation has already been performed by Faltings \cite[\S 4]{Faltings_2005} (for curves) and \cite[Proposition 3.3.22]{abbes2026twistinghiggsmodulesfunctorial}.

We first fix notation for Higgs bundles described by descent data.

\begin{convention}
    Let $\frakX$ be a smooth formal scheme over $\calO _{\mathbf{C}}$ with generic fiber $X$. 
    Let $\{\frakU _{\lambda}=\Spf(R_{\lambda})\}_{\lambda \in \Lambda}$ be an affine open covering of $\frakX$. 
    Assume that there are $a$-Hitchin small Higgs $R_{\lambda}$-modules $(M_{\lambda},\theta_{\lambda})$ for all $\lambda \in \Lambda$, 
    and isomorphisms of Higgs modules 
    \[\varphi_{\lambda\mu} : R_{\lambda\mu}\otimes_{R_{\lambda}} (M_{\lambda},\theta_{\lambda}) \cong  R_{\lambda\mu}\otimes_{R_{\mu}} (M_{\mu},\theta_{\mu})\] 
    for all $\lambda,\mu \in \Lambda$, satisfying the following cocycle condition 
    \[ 1 \otimes \varphi_{\lambda \nu}= (1 \otimes \varphi_{\mu \nu})\circ (1 \otimes \varphi_{\lambda \mu}), \quad \forall \lambda,\mu,\nu \in \Lambda.\] 
    Then there is, up to unique isomorphism, a unique $a$-Hitchin-small Higgs bundle $(\mathcal{H},\theta)$ on $\frakX$, together with isomorphisms 
    \[ \varphi_{\lambda} : \Gamma(\frakU _{\lambda},(\mathcal{H},\theta)) \cong  (M_{\lambda},\theta_{\lambda}),\]
    such that 
    \[ 1 \otimes \varphi_{\mu} = \varphi_{\lambda\mu} \circ (1\otimes \varphi_{\lambda} ).\]
    We say that the Higgs bundle $(\mathcal{H},\theta)$ is given by the descent datum $((M_{\lambda},\theta_{\lambda}),\varphi_{\lambda\mu})$. 
\end{convention}

\begin{convention}
    Let $\frakX$ be a  smooth formal scheme over $\calO _{\mathbf{C}}$, and let $(\mathcal{H},\theta)$ be a Higgs bundle on $\frakX$. 
    Then the Higgs field $\theta : \calH \to \calH \otimes \rho \widehat{\Omega}_{\frakX}(-1)$ induces a map
    $$\rho^{-1}T_{\frakX}(1)  \xrightarrow{\mathrm{coev} \otimes 1} \mathcal{H}^\vee \otimes \mathcal{H} \otimes \rho^{-1} T_{\frakX}(1) \xrightarrow{1 \otimes \theta \otimes 1} \mathcal{H}^\vee \otimes \mathcal{H} \otimes \rho\widehat{\Omega}_{\frakX}(-1) \otimes \rho^{-1}T_{\frakX}(1)  \xrightarrow{1 \otimes \mathrm{ev}} \mathcal{H}^\vee \otimes \mathcal{H} \cong \EEnd(\mathcal{H}),$$
    which is also denoted by $\theta$ by abuse of notation. 
    The condition $\theta \wedge \theta = 0$ implies that the map 
    $\theta : \rho^{-1} T_{\frakX}(1) \to \EEnd(\mathcal{H})$ factors through the endomorphism sheaf $\EEnd(\mathcal{H},\theta)$ of the Higgs bundle $(\calH,\theta)$. 
    
    Similarly,  let $X$ be a  smooth rigid space over $\mathbf{C}$ and let $(\mathcal{H},\theta)$ be a Higgs bundle on $X$. 
    Then the Higgs field $\theta$ induces a map 
    \[ \theta : T_{X}(1) \to \EEnd(\mathcal{H},\theta).\]
\end{convention}

Recall that by standard deformation theory, when $\frakX$ admits a flat $A_{2}$-lifting, the isomorphism classes of flat $A_{2}$-liftings of $\frakX$ form a torsor under the first cohomology group $\mathrm{H}^1(\frakX,\rho^{-1}T_{\frakX}(1))$. 

\begin{lemma}\label{lem:action-of-lift}
    Let $\frakX$ be a smooth formal scheme over $\calO _{\mathbf{C}}$, 
    and let $\widetilde{\frakX}_{1}$ and $\widetilde{\frakX}_{2}$ be two flat $A_{2}$-liftings of $\frakX$. 
    Assume that the difference class $[\widetilde{\frakX}_{2}]-[\widetilde{\frakX}_{1}]$ is represented by the Čech $1$-cocycle $(a_{\lambda\mu}=\frac{t}{\rho} \delta_{\lambda\mu})$ valued in $\rho^{-1}T_{\frakX}(1)$. 

    Choose a covering $\{\frakU _{\lambda}\}$ of $\frakX$ consisting of small affine open subsets. 
    Let $(\mathcal{H},\theta)$ be an $a$-Hitchin small Higgs bundle on $\frakX$ given by the descent datum $((M_{\lambda},\theta_{\lambda}),\varphi_{\lambda\mu})$.
    
    Then the Higgs bundle $\mathfrak{S}_{\widetilde{\frakX}_{2}}(\mathfrak{S}^{-1}_{\widetilde{\frakX}_{1}}(\mathcal{H},\theta))$ is given by the descent datum
    \[ ((M_{\lambda},\theta_{\lambda}),\varphi_{\lambda\mu}\circ \exp(\rho\theta(a_{\lambda\mu}))). \]
\end{lemma}
\begin{proof}
    As a smooth affine formal scheme admits a unique $A_{2}$-lift up to isomorphism, 
    the Higgs modules $(M_{\lambda},\theta_{\lambda})$ are unchanged, 
    and we only need to compute the transition functions gluing $\mathfrak{S}_{\widetilde{\frakX}_{2}}(\mathfrak{S}^{-1}_{\widetilde{\frakX}_{1}}(\mathcal{H},\theta))$. 
    
    Let $\widetilde{\frakU }_{\lambda}=\widetilde{\frakX}_{1}\big|\frakU _{\lambda}$ 
    and $\widetilde{\frakU }_{\lambda\mu}=\widetilde{\frakX}_{1}\big|\frakU _{\lambda\mu}$, where $\frakU _{\lambda\mu}=\frakU _{\lambda}\cap \frakU _{\mu}$. 
    Assume that $\widetilde{\frakU }_{\lambda}=\Spf(\widetilde{R}_{\lambda})$ and $\widetilde{\frakU }_{\lambda\mu}=\Spf(\widetilde{R}_{\lambda\mu})$. 
    Then $\widetilde{\frakX}_{2}$ is obtained by gluing those $\widetilde{\frakU }_{\lambda}$ along the overlaps $\widetilde{\frakU }_{\lambda\mu}$ 
    via the transition functions $\widetilde{\frakU }_{\lambda\mu}\cong  \widetilde{\frakU }_{\mu\lambda}$ induced by 
    \[ \mathrm{id} + \frac{t}{\rho}\delta_{\lambda\mu} : \widetilde{R}_{\lambda\mu} \cong  \widetilde{R}_{\mu\lambda}.\]
    The computation of transition functions reduces to the following calculation:

    Let $\frakU  = \Spf(R)$ be a small affine smooth formal scheme over $\mathcal{O}_{\mathbf{C}}$, 
    let $\widetilde{\frakU }=\Spf(\widetilde{R})$ be an $A_{2}$-lifting of $\frakU $, 
    and let $\delta \in \operatorname{Der}_{\calO _{\mathbf{C}}}(R)$ be a (continuous) derivation. 
    Consider the automorphism of $\calO \widehat{\mathbb{C}}^{+}_{\mathrm{pd}}$ on $U_{\infty}$ 
    induced by the automorphism $\mathrm{id} + \frac{t}{\rho}\delta : \widetilde{R} \cong  \widetilde{R}$. 

    By Lemma \ref{lem:local-description-of-int-Faltings-ext}, the choice of a toric chart induces a splitting 
    \[ \mathcal{E}^+_{\widetilde{\frakU }}|_{U_{\infty}} \cong \widehat{\calO }^+_{U_{\infty}}e \oplus \left(\bigoplus_{i=1}^{d} \widehat{\calO }^+_{U_{\infty}}\rho y_{i}\right),\]
    where $\rho y_{i}$ are sent to $\frac{\rho\mathrm{d}\log T_{i}}{t}$. 
    The automorphism of $\mathcal{E}^+_{\widetilde{\frakU }}|_{U_{\infty}}$ induced by $\mathrm{id} + \frac{t}{\rho}\delta$ is given by  
    \[\rho y_{i} \mapsto \rho y_{i} + \delta(\mathrm{d}\log T_{i}) e.\] 
    By Lemma \ref{lem:local-description-of-OCpd}, we obtain an isomorphism 
    \[ \calO \widehat{\mathbb{C}}^+_{\mathrm{pd}}|_{U_{\infty}} \cong  \widehat{\calO }^+_{U}[\rho Y_{1},\cdots, \rho Y_{d}]_{\mathrm{pd}}^{\wedge}|_{U_{\infty}}\]
    sending $\rho y_{i}$ to $\rho Y_{i}$. 
    Consequently, the automorphism of $\widehat{\calO }^+_{U}[\rho Y_{1},\cdots,\rho Y_{d}]_{\mathrm{pd}}^{\wedge}|_{U_{\infty}}$ 
    induced by $\mathrm{id} + \frac{t}{\rho}\delta$ is given by  
    \[\rho Y_{i} \mapsto \rho Y_{i} + \rho\delta(\mathrm{d}\log T_{i}).\] 

    We now translate this coordinate change into its effect on a Higgs module. 
    Let $(H,\theta)$ be an $a$-Hitchin small Higgs $R$-module. 
    By the proof of \cite[Theorem 4.11]{MinWang2024}, 
    there exists an $a$-Hitchin small $R$-representation $V$ of $\Gamma$ and commuting $a$-Hitchin small endomorphisms $\Theta_{i}$ of $V$, 
    such that 
    \[ H= \exp\left(\sum_{i=1}^{d}\Theta_{i}Y_{i} \right)(V) \subset V \otimes_{R} R[\rho Y_{1},\cdots,\rho Y_{d}]^{\wedge}_{\mathrm{pd}},\]
    and 
    \[ \theta = \sum_{i=1}^{d} \Theta_{i} \otimes \frac{\mathrm{d}\log T_{i}}{t}.\]
    Thus the change of variables $\rho Y_{i} \mapsto \rho Y_{i} + \rho\delta(\mathrm{d}\log T_{i})$ 
    induces an automorphism of $V \otimes_{R} R[\rho Y_{1},\cdots,\rho Y_{d}]^{\wedge}_{\mathrm{pd}}$, which sends $H$ to
    \[ \exp\left(\sum_{i=1}^{d}\Theta_{i}(Y_{i}+ \delta(\mathrm{d}\log T_{i}))\right)(V) 
    = \exp\left(\sum_{i=1}^{d} \delta(\mathrm{d}\log T_{i}) \Theta_{i}\right)(H)
    = \exp\left(\rho\theta(\frac{t}{\rho}\delta) \right)(H).\]
    So we obtain the desired transition functions.
\end{proof}

The local calculation in the integral case has a rational analogue, which we record although it will not be used in the proof below.

\begin{remark}\label{rmk:rat-transition}
    Let $X$ be a proper smooth rigid space over $\bfC$. 
    Assume that there are smooth $B_{2}$-liftings $\widetilde{X}_{1}$ and $\widetilde{X}_{2}$ of $X$. 
    Let $(\calH,\theta)$ be a nilpotent Higgs bundle on $X$.  
    Assume that there is a covering $\{U_{\lambda}\}$ of $X$, 
    consisting of affinoid opens, 
    and $(\calH,\theta)$ is given by the descent datum $((M_{\lambda},\theta_{\lambda}),\varphi_{\lambda\mu})$. 
    
    Assume that $[\widetilde{X}_{2}]-[\widetilde{X}_{1}]$ is represented by the Čech $1$-cocycle $(a_{\lambda\mu})$ valued in $T_{X}(1)$. Fix an exponential $\Exp$ of $\mathbf{C}$.  
    Then a similar calculation shows that the Higgs bundle $S_{\Exp,\widetilde{X}_{2}}^{-1}(S_{\Exp,\widetilde{X}_{1}}(\mathcal{H},\theta))$ is given by the descent datum
    \[ ((M_{\lambda},\theta_{\lambda}),\varphi_{\lambda\mu}\circ \exp(\theta(a_{\lambda\mu}))). \] 
    Here $S_{\Exp,\widetilde{X}_{i}}$ is the $p$-adic Simpson functor constructed in \cite{heuer2023padic}. 
\end{remark}

We now turn the preceding local calculation into the global notions of lift-independence.

\begin{definition}\label{def:int-lift-indep-Higgs}
    Let $\frakX$ be a smooth formal scheme over $\calO_{\bfC}$ that admits a flat $A_{2}$-lifting. 
    We say that a $\frac{1}{p-1}$-Hitchin small Higgs bundle $(\calH, \theta)$ on $\frakX$ is \textit{integrally lift-independent} if for any two flat $A_2$-liftings $\widetilde{\frakX}_1$ and $\widetilde{\frakX}_2$ of $\frakX$, there exists an isomorphism 
    \[ \mathfrak{S}_{\widetilde{\frakX}_1}^{-1}(\calH, \theta) \cong  \frakS_{\widetilde{\frakX}_2}^{-1}(\calH, \theta) \]
    of $\frac{1}{p-1}$-Hitchin small $\widehat{\calO}_X^+$-representations on $X_{\text{proét}}$. 
\end{definition}

\begin{definition}\label{def:rat-lift-indep-Higgs}
    Let $X$ be a smooth proper rigid space over $\bfC$ and fix an exponential $\operatorname{Exp}$ for $\bfC$. 
    We say that a Higgs bundle $(\calH, \theta)$ on $X$ is \textit{lift-independent} if for any two flat $B_2$-liftings $\widetilde{X}_1$ and $\widetilde{X}_2$ of $X$, there is an isomorphism 
    \[ S_{\Exp,\widetilde{X}_1}(\calH, \theta) \cong  S_{\Exp,\widetilde{X}_2}(\calH, \theta) \]
    of pro-étale vector bundles on $X$.
\end{definition}

\begin{remark}
    By \cite[Proposition 7.4.4]{Guo2023}, a smooth proper rigid space over $\bfC$ always admits a flat $B_{2}$-lifting. 
\end{remark}

The following two lemmas provide the key technical tool for this section, showing how an isomorphism between Higgs fields on a single fiber extends to a neighborhood in the base.

\begin{lemma}\label{lem:fiberwise_trivial_vector_bundle_rigid}
    Let $X$ be a proper rigid variety over $\bfC$, and let $S$ be a reduced
    rigid variety over $\bfC$. Let $\calE$ be a vector bundle on $X$ and
    $\calF$ a vector bundle on $X\times S$.
    Assume that $\calF(s)\cong\calE$ for every $s\in S(\bfC)$.
    Then, for every $s_0\in S(\bfC)$ and every isomorphism
    $\varphi_0:\calE\xrightarrow{\sim}\calF(s_0)$, there are an open
    neighborhood $U$ of $s_0$ and an isomorphism
    \[
        \varphi:(p^*\calE)|_{X\times U}\xrightarrow{\sim}\calF|_{X\times U}
    \]
    extending $\varphi_0$, where $p:X\times S\to X$ is the projection.
\end{lemma}

\begin{proof}
    Let $q:X\times S\to S$ be the other projection and set
    $\calH=\HHom(p^*\calE,\calF)$. Since $q$ is proper and flat and $\calH$
    is locally free, $\calH$ is $S$-flat. Moreover,
    \[
        h^0(X,\calH(s))=h^0(X,\EEnd(\calE))
    \]
    for every $s\in S(\bfC)$. Grauert's theorem
    \cite[Proposition~3.15]{Heuer_HodgeTatespectralsequence} therefore shows
    that $q_*\calH$ is locally free and that
    \[
        (q_*\calH)(s)\xrightarrow{\sim}
        \rH^0(X,\calH(s))=\Hom(\calE,\calF(s)).
    \]
    Hence $\varphi_0$ lifts, after shrinking around $s_0$, to a section of
    $q_*\calH$, and thus to a morphism
    $\varphi:(p^*\calE)|_{X\times U}\to\calF|_{X\times U}$.
    Its non-isomorphism locus $Z\subset X\times U$ is a closed analytic subset
    disjoint from $X\times\{s_0\}$. Since $q$ is proper, $q(Z)$ is closed;
    replacing $U$ by $U\setminus q(Z)$ makes $\varphi$ an isomorphism.
\end{proof}

\begin{definition}
    Let $X$ be a smooth rigid variety over $\bfC$, let $S$ be a rigid variety
    over $\bfC$, and let $p:X\times S\to X$ be the projection. An $S$-family
    of Higgs bundles on $X$ is a pair $(\calF,\Theta)$ consisting of a vector
    bundle $\calF$ on $X\times S$ and an $\calO_{X\times S}$-linear morphism
    $\Theta:\calF\to\calF\otimes p^*\Omega_X^1(-1)$ satisfying
    $\Theta\wedge\Theta=0$.
\end{definition}

\begin{lemma}\label{lem:fiberwise_trivial_Higgs_bundle_rigid}
    Let $X$ be a proper smooth rigid variety over $\bfC$, and let $S$ be a
    smooth rigid variety over $\bfC$. Let $(\calE,\theta)$ be a Higgs bundle
    on $X$ and $(\calF,\Theta)$ an $S$-family of Higgs bundles on $X$.
    Assume that $(\calF(s),\Theta(s))\cong(\calE,\theta)$ for every
    $s\in S(\bfC)$. Then, for every $s_0\in S(\bfC)$ and every isomorphism
    $\varphi_0:(\calE,\theta)\xrightarrow{\sim}(\calF(s_0),\Theta(s_0))$,
    there are an open neighborhood $U$ of $s_0$ and an isomorphism
    $\varphi:p^*(\calE,\theta)|_{X\times U}\xrightarrow{\sim}
    (\calF,\Theta)|_{X\times U}$ extending $\varphi_0$, where
    $p:X\times S\to X$ is the projection.
\end{lemma}

\begin{proof}
    Let $q:X\times S\to S$ be the other projection and put
    \[
        \mathcal C^0=\HHom(p^*\calE,\calF),\qquad
        \mathcal C^1=\mathcal C^0\otimes p^*\Omega_X^1(-1).
    \]
    The Higgs fields induce
    \[
        \delta:\mathcal C^0\longrightarrow\mathcal C^1,\qquad
        f\longmapsto\Theta\circ f-(f\otimes1)\circ p^*\theta.
    \]
    For any $s\in S(\bfC)$, $\mathcal C^0(s)\cong\EEnd(\calE)$ and
    $\mathcal C^1(s)\cong\EEnd(\calE)\otimes\Omega_X^1(-1)$. Their $h^0$
    are therefore independent of $s$. By Grauert's theorem,
    \[
        \mathcal A=q_*\mathcal C^0,\qquad \mathcal B=q_*\mathcal C^1
    \]
    are vector bundles, with natural isomorphisms
    $\mathcal A(s)\xrightarrow{\sim}\rH^0(X,\mathcal C^0(s))$ and
    $\mathcal B(s)\xrightarrow{\sim}\rH^0(X,\mathcal C^1(s))$ for every
    $s\in S(\bfC)$. Let $D=q_*\delta:\mathcal A\to\mathcal B$.
    Under any isomorphism $(\calF(s),\Theta(s))\cong(\calE,\theta)$,
    $D(s)$ identifies with
    \[
        [\theta,-]:\rH^0(X,\EEnd(\calE))\longrightarrow
        \rH^0\bigl(X,\EEnd(\calE)\otimes\Omega_X^1(-1)\bigr),
    \]
    so $\dim_{\bfC}\ker D(s)$ is independent of $s$. Since $S$ is reduced,
    this implies that $\mathcal K=\ker D$ is locally
    free, with natural isomorphisms
    \[
        \mathcal K(s)\xrightarrow{\sim}
        \ker D(s)\cong
        \Hom\bigl((\calE,\theta),(\calF(s),\Theta(s))\bigr).
    \]
    By left exactness, $\mathcal K=q_*(\ker\delta)$. Thus $\varphi_0$ lifts
    locally to a section of $\mathcal K$, giving a Higgs morphism
    $p^*(\calE,\theta)|_{X\times U}\to(\calF,\Theta)|_{X\times U}$.
    The non-isomorphism-locus argument above makes it an isomorphism after
    shrinking $U$.
\end{proof}

\begin{theorem}\label{theo: int lift independ to quasi independ}
    Let $\frakX$ be a proper smooth formal scheme over $\calO_{\bfC}$
    admitting a flat $A_2$-lifting, and let $(\calH,\theta)$ be a
    $\frac{1}{p-1}$-Hitchin-small Higgs bundle on $\frakX$.
    If $(\calH,\theta)$ is integrally lift-independent, then
    \[
        \rH^1(\frakX,\rho^{-1}T_{\frakX}(1))\Bigl[\frac1p\Bigr]
        \xrightarrow{\,\rH^1(\theta)[1/p]\,}
        \rH^1(\frakX,\EEnd(\calH,\theta))\Bigl[\frac1p\Bigr]
    \]
    is zero. Equivalently, under the proper formal--generic-fiber comparison,
    \[
        \rH^1(X,T_X(1))\xrightarrow{\,\rH^1(\theta)\,}
        \rH^1\bigl(X,\EEnd(\calH[1/p],\theta)\bigr)
    \]
    is zero, where $X$ and $(\calH[1/p],\theta)$ are the generic fibers of
    $\frakX$ and $(\calH,\theta)$, respectively.
\end{theorem}

\begin{proof}
    Put $\mathcal T=\rho^{-1}T_{\frakX}(1)$. Choose a finite small affine cover
    $\{\frakU_i\}_{i\in I}$ of $\frakX$ and write $(\calH,\theta)$ as the
    descent datum $((M_i,\theta_i),g_{ij})$. This cover computes the
    coherent cohomology of $\mathcal T$. Thus, for any
    $\eta\in\rH^1(\frakX,\mathcal T)[1/p]$, there are $N\geq0$ and a
    \v{C}ech $1$-cocycle $(D_{ij})$ for $\mathcal T$ on $\{\frakU_i\}$
    representing $p^N\eta$. Let $\{U_i\}$ be the induced cover of $X$ and
    $\bD^1=\Spa(\bfC\langle z\rangle)$. On $X\times\bD^1$, consider the
    descent datum
    \[
        \left(
        \bigl(M_i[1/p]\langle z\rangle,\theta_i\bigr),
        \exp\bigl(\rho z\theta_j(D_{ij})\bigr)\circ g_{ij}
        \right).
    \]
    Hitchin-smallness gives convergence of the exponential, while
    $\theta\wedge\theta=0$ and the cocycle condition for $(D_{ij})$ give the
    cocycle condition for the transition maps. Hence this datum defines a
    $\bD^1$-family $(\calF,\Theta)$ of Higgs bundles on $X$.

    Fix a flat $A_2$-lifting $\widetilde{\frakX}_0$. For every
    $z\in\bD^1(\bfC)=\calO_{\bfC}$, the lifting torsor provides a flat lifting
    $\widetilde{\frakX}_z$ whose difference from $\widetilde{\frakX}_0$ is
    $z[(D_{ij})]$. Lemma~\ref{lem:action-of-lift} and integral
    lift-independence give
    \[
        (\calF(z),\Theta(z))\cong
        \bigl(\frakS_{\widetilde{\frakX}_z}\frakS_{\widetilde{\frakX}_0}^{-1}
        (\calH,\theta)\bigr)[1/p]
        \cong(\calH[1/p],\theta).
    \]
    Applying Lemma~\ref{lem:fiberwise_trivial_Higgs_bundle_rigid} at $z=0$
    gives, on a neighborhood of $0$, an isomorphism
    \[
        \varphi:p_X^*(\calH[1/p],\theta)\xrightarrow{\sim}(\calF,\Theta),
        \qquad \varphi(0)=\id,
    \]
    where $p_X:X\times\bD^1\to X$ is the projection.

    Restrict to $\Spa(\bfC\langle z\rangle/(z^2))$ and write
    $\varphi|_{U_i}=\id+z\psi_i$. Since $\varphi$ is a Higgs isomorphism,
    $\psi_i\in\Gamma(U_i,\EEnd(\calH[1/p],\theta))$. Compatibility with the
    transition maps gives
    \[
        (\id+z\psi_j)\circ g_{ij}
        =(\id+\rho z\theta_j(D_{ij}))\circ g_{ij}\circ(\id+z\psi_i).
    \]
    Comparing coefficients of $z$ yields
    \[
        \rho\theta_j(D_{ij})
        =\psi_j-g_{ij}\circ\psi_i\circ g_{ij}^{-1}.
    \]
    Thus $(\rho\theta(D_{ij}))$ is a \v{C}ech coboundary, so
    \[
        0=\rho\,\rH^1(\theta)(p^N\eta)
        =\rho p^N\rH^1(\theta)(\eta).
    \]
    Since $\rho p^N\in\bfC^\times$, $\rH^1(\theta)(\eta)=0$. As $\eta$ was
    arbitrary, the map is zero.
\end{proof}

\section{Small Rank Semistable Cohomologically Lift-independent Higgs Bundles Have Zero Higgs Fields}\label{section: indep to 0}

Throughout this section, let $K$ be an algebraically closed field of characteristic $0$.
By a variety over $K$, we mean a rigid analytic variety when $K$ is a complete $p$-adic field, and an algebraic variety otherwise.
Proper rigid analytic curves are algebraic \cite[\S1.3]{Conrad2006ModularCurves}.
Throughout this section, we therefore regard proper curves as algebraic curves, identifying their coherent sheaves and cohomology via rigid GAGA \cite[Example~3.2.6]{Conrad2006RelativeAmpleness} and passing to analytifications in the $p$-adic arguments.
The fiber notation $\calF(x)$ and $f(x)$ is used in both settings.
The notions of lift-independence and integral lift-independence are used only over $\bfC$.
When $K=\bfC$, fix a choice of $\epsilon=(1,\zeta_p,\zeta_{p^2},\dots)\in\bfC^{\flat}$, and thus identify $\bfC(1)\cong\bfC$.

Recall the following definition:

\begin{definition}[Cohomological lift-independence]
    Let $X$ be a smooth variety over $K$. A Higgs bundle $(\calH,\theta)$ on $X$ is called cohomologically lift-independent if the map
    \[\rH^1(X,T_{X}) \xrightarrow[]{\rH^1(\theta)} \rH^1(X,\EEnd(\calH,\theta))\]
    is zero.
\end{definition}

\subsection{General facts of (cohomologically) lift-independent Higgs bundles}

Throughout this subsection, let $X$ be a smooth projective connected curve over $K$ of genus $g\geqslant2$.

We begin by giving an example of a lift-independent Higgs bundle $(\calH,\theta)$ that is nilpotent but $\theta\neq 0$.

\begin{example}\label{eg: without semi stable}
    Let $K=\bfC$. 
    Let $\calL$ be a line bundle of degree $d > 2g - 2$ on $X$. 
    Let $s \in  \rH^0(X, \calL \otimes \Omega_{X})$ be a nonzero global section of $\mathcal{L}\otimes \Omega_{X}$. 
    Let $\calH = \calL \oplus \calO_{X}$, and let
    $$\theta = \begin{pmatrix} 0 & s \\ 0 & 0 \end{pmatrix} : \calH \xrightarrow[]{\mathrm{pr}} \calO_{X} \xrightarrow[]{s} \calL \otimes \Omega_{X}  \xrightarrow[]{\mathrm{incl}} \calH \otimes \Omega_{X} .$$
    Then $(\calH, \theta)$ is lift-independent.

    Indeed, 
    $$\EEnd(\calH, \theta) = \calO_{X} \oplus \calL.$$
    The Higgs field $\theta$ induces a map
    $$\theta : T_{X}  \xrightarrow[]{\mathrm{coev} \otimes 1} \calH^\vee \otimes \calH \otimes T_X  \xrightarrow[]{1 \otimes \theta \otimes 1}\calH^\vee \otimes \calH \otimes \Omega_{X}  \otimes T_{X}  \xrightarrow[]{1 \otimes \mathrm{ev}} \calH^\vee \otimes \calH = \EEnd(\calH),$$
    whose image is contained in $\EEnd(\calH, \theta)$. 
    By the definition of $(\calH, \theta)$, the composition
    $$\calO_{X} \to \calO_{X} \oplus \calL^\vee = \EEnd(\calH, \theta)^\vee \xrightarrow{\theta^\vee} \Omega_X $$
    is zero. 
    As $\deg(\calL^\vee \otimes \Omega_{X}) < 0$, we see that
    $$ \rH^0(X, \calL^\vee \otimes \Omega_{X}) = 0.$$
    Hence
    $$ \rH^0(X,\EEnd(\calH, \theta)^\vee \otimes \Omega_{X}) \xrightarrow[]{ \rH^0(\theta^\vee)}  \rH^0(X, \Omega_{X}^{\otimes 2} )$$
    is zero. 
    By Serre duality,
    $$ \rH^1(X, T_X ) \xrightarrow[]{ \rH^1(\theta)}  \rH^1(X, \EEnd(\calH, \theta))$$
    is zero.

    Fix an affinoid cover $\frakU = \{U_1, U_2\}$ of $X$. 
    For any $a \in \check{Z}^1(\mathfrak{U}, T_X )$, we know that $\theta(a) \in \check{B}^1(\mathfrak{U}, \EEnd(\calH, \theta))$, 
    that is, $\exists \psi_i = \psi_i(a) \in \Gamma(U_i, \EEnd(\calH, \theta))$, such that
    $$\theta(a) = \psi_2|_{U_{12}} - \psi_1|_{U_{12}}$$ where $U_{12}=U_1\cap U_2$.
    By the construction of $(\calH, \theta)$, we see that the composition
    $$T_X  \xrightarrow[]{\theta} \EEnd(\calH, \theta) = \calO_{X} \oplus \calL \to \calO_{X}$$
    is zero, so we may replace $\psi_i$ by their image under
    $$\Gamma(U_{i}, \EEnd(\calH, \theta)) \xrightarrow[]{\mathrm{pr}} \Gamma(U_{i}, \calL) \xrightarrow[]{\mathrm{incl}} \Gamma(U_{i}, \EEnd(\calH, \theta)).$$
    Thus we may assume that $\psi_i$ are nilpotent.
    Let $\varphi_i = \exp(\psi_i)$. We obtain a commutative diagram
    $$
    \begin{tikzcd}
        \calH|_{U_{12}} \arrow[r, "\mathrm{id}"] \arrow[d, "\varphi_1"] & \mathcal{H}|_{U_{12}} \arrow[d, "\varphi_2"'] \\
        \calH|_{U_{12}} \arrow[r, "\exp(\theta(a))"'] & \mathcal{H}|_{U_{12}}
    \end{tikzcd}
    $$
    which implies that $(\calH, \theta)$ is lift-independent by Remark \ref{rmk:rat-transition}.
\end{example}

\begin{lemma}\label{lemma: Basic lemma}
    Any cohomologically lift-independent Higgs bundle $(\calH,\theta)$ satisfies $\tr(\theta) = 0$ and $\tr(\theta^2) = 0$. 
    In particular, if $\rk(\calH)=1$, $\theta=0$; if $\rk(\calH) = 2$, $\theta^2=0$.
\end{lemma}
\begin{proof}
    Since $(\calH, \theta)$ is cohomologically lift-independent, the map $\rH^1(X, T_{X}) \xrightarrow[]{\rH^1(\theta)} \rH^1(X, \EEnd(\calH))$ is zero. 
    By Serre duality, it follows that the map 
    $$ \rH^0(X, \EEnd(\calH) \otimes \Omega_X) \xrightarrow[]{\operatorname{tr}(-\circ \theta)} \rH^0(X, \Omega_{X}^{\otimes 2}) $$
    is also zero. 
    Applying this to $\theta \in \rH^0(X, \EEnd(\calH) \otimes \Omega_{X})$, we conclude that $\operatorname{tr}(\theta^2) = 0$.

    Given $g \geqslant 2$, we have $\rH^0(X, \Omega_{X}) \neq 0$. Choose a nonzero $\eta \in \mathrm{H}^0(X, \Omega_{X})$. 
    We then obtain the following commutative diagram:
    $$
    \begin{tikzcd}
        \rH^0(X, \EEnd(\calH)) \arrow[r, "\tr(- \circ \theta)"] \arrow[d, "1 \otimes \eta"'] 
        &  \rH^0(X, \Omega_{X}) \arrow[d, "1 \otimes \eta"] \\
        \rH^0(X,\EEnd(\calH) \otimes \Omega_{X}) \arrow[r, "\tr(- \circ \theta)"'] 
        &  \rH^0(X, \Omega_{X}^{\otimes 2}).
    \end{tikzcd}
    $$
    Since $\eta \neq 0$, the map $1 \otimes \eta: \Omega_X \to \Omega_X^{\otimes 2}$ is injective. 
    Considering global sections, the induced map $1 \otimes \eta: \rH^0(X, \Omega_{X}) \to \rH^0(X, \Omega_{X}^{\otimes 2})$ remains injective. 
    Consequently, the composition $\tr(-\circ \theta) \circ (1 \otimes \eta)$ is the zero map, implying that the map $\rH^0(X, \EEnd(\calH)) \xrightarrow[]{\tr(-\circ \theta)} \rH^0(X, \Omega_{X})$ must be zero. 
    Applying this to the identity section $\id \in \rH^0(X, \EEnd(\calH))$, we conclude that $\tr(\theta) = 0$.
\end{proof}

\begin{definition}
    A closed point $x \in X$ is called a Weierstrass point if
    $$ h^0(X, \calO_{X}(gx)) \geq 2.$$

    There are only finitely many Weierstrass points
    \cite[Proposition~3(ii) and Theorem~11]{Laksov1981}.
    If $x$ is a non-Weierstrass point, we have
    $$ h^0(X, \calO_{X}(mx)) = 1, \quad \forall 1 \leq m \leq g.$$
\end{definition}

\begin{lemma}\label{lem:extension_lemma}
    Let $(\calH,\theta)$ be a cohomologically lift-independent Higgs bundle on $X$.
    Let $x\in X$ be a closed point and put $\calL=\calO_X(x)$.
    For any isomorphism $\alpha : \Omega_X(x) \xrightarrow{\sim} \calL(x)$,
    there exists a morphism $\Phi: \calH \to \calH \otimes \calL$ such that
    $$ \Phi(x)= (1 \otimes \alpha) \circ \theta(x).$$
\end{lemma}
\begin{proof}
    By cohomological lift-independence, the composite
    $$\rH^1(X,T_X)\xrightarrow{0}\rH^1(X,\EEnd(\calH,\theta))
      \longrightarrow\rH^1(X,\EEnd(\calH))$$
    induced by $\theta$ is zero.
    Consider the canonical section $s \in \rH^0(X,\calL)$, whose zero divisor is $x$.
    Let $i : \{x\}=\Spec K \hookrightarrow X$ be the inclusion.
    There are short exact sequences
    $$
    \begin{tikzcd}
        0 \arrow[r]
        & T_X \arrow[r, "s"] \arrow[d,"\theta"']
        & T_X \otimes \calL \arrow[r] \arrow[d,"\theta \otimes 1"]
        & i_*((T_X \otimes \calL)(x)) \arrow[r] \arrow[d, "i_*((\theta \otimes 1)(x))"]
        & 0 \\
        0 \arrow[r]
        & \EEnd(\calH) \arrow[r, "s"]
        & \EEnd(\calH) \otimes \calL \arrow[r]
        & i_*((\EEnd(\calH) \otimes \calL)(x)) \arrow[r]
        & 0.
    \end{tikzcd}
    $$
    Taking cohomology, we obtain a commutative diagram
    \begin{equation}\label{diag:lemma_extension}
    \begin{tikzcd}
        \rH^0(X, T_X \otimes \calL) \arrow[r] \arrow[d,"\rH^0(\theta \otimes 1)"']
        &(T_X \otimes \calL)(x) \arrow[r, "\delta_T"] \arrow[d, "(\theta \otimes 1)(x)"']
        & \rH^1(X, T_X) \arrow[d, "\rH^1(\theta)=0"] \\
        \rH^0(X, \EEnd(\calH) \otimes \calL) \arrow[r]
        &(\EEnd(\calH) \otimes \calL)(x) \arrow[r, "\delta_E"]
        & \rH^1(X, \EEnd(\calH))
    \end{tikzcd}
    \end{equation}
    with exact rows. Via the canonical identification
    $$\Hom_K(\Omega_X(x),\calL(x))\cong(T_X\otimes\calL)(x),$$
    we regard $\alpha$ as an element of $(T_X\otimes\calL)(x)$.
    Under this identification, $(\theta\otimes1)(x)$ sends $\alpha$ to
    $(1\otimes\alpha)\circ\theta(x)$.
    Since
    $$\delta_E\bigl((\theta\otimes1)(x)(\alpha)\bigr)
      =\rH^1(\theta)(\delta_T(\alpha))=0,$$
    we may extend $(1\otimes\alpha)\circ\theta(x)$ to a global section
    $$\Phi \in \rH^0(X,\EEnd(\calH)\otimes\calL)
      =\Hom(\calH,\calH\otimes\calL)$$
    by the exactness of the bottom row of diagram \ref{diag:lemma_extension}.
\end{proof}

\begin{proposition}\label{prop:vanish_of_power_trace_of_CLI}
    Let $(\calH,\theta)$ be a cohomologically lift-independent Higgs bundle on $X$. Then
    $$ \tr(\theta^m) = 0 , \quad \forall 1 \leq m \leq g.$$
\end{proposition}
\begin{proof}
    Let $x \in X$ be a non-Weierstrass closed point and put $\calL=\calO_X(x)$.
    Consider the canonical section $s \in \rH^0(X,\calL)$, whose zero divisor is $x$.
    Since $x$ is non-Weierstrass,
    $$\rH^0(X,\calL^{\otimes m})=K\cdot s^{\otimes m},\quad 1\leq m\leq g.$$

    Choose an isomorphism $\alpha : \Omega_X(x) \xrightarrow{\sim} \calL(x)$.
    By Lemma \ref{lem:extension_lemma}, there exists a morphism $\Phi: \calH\to\calH\otimes\calL$ such that
    $$ \Phi(x)=(1\otimes\alpha)\circ\theta(x).$$
    For $1\leq m\leq g$, the section $\tr(\Phi^m)\in\rH^0(X,\calL^{\otimes m})$
    is a scalar multiple of $s^{\otimes m}$. Since $s(x)=0$, we have
    $$0=\tr(\Phi^m)(x)=\alpha^{\otimes m}\bigl(\tr(\theta^m)(x)\bigr).$$
    As $\alpha^{\otimes m}$ is an isomorphism, we deduce that $\tr(\theta^m)(x)=0$.

    Thus $\tr(\theta^m)$ vanishes at every non-Weierstrass closed point.
    The non-Weierstrass closed points are Zariski dense in $X$, since there are only finitely many Weierstrass points.
    Hence $\tr(\theta^m)=0$.
\end{proof}

\begin{theorem}\label{theo:rank_le_genus_nilpotent}
    Any cohomologically lift-independent Higgs bundle $(\calH,\theta)$ on $X$
    of rank $r\leqslant g$ is nilpotent; more precisely, $\theta^r=0$.
\end{theorem}
\begin{proof}
    By Proposition \ref{prop:vanish_of_power_trace_of_CLI}, we have
    $\tr(\theta^m)=0$ for $1\leqslant m\leqslant r$.
    After locally trivializing $\Omega_X$, Newton's identities in characteristic zero
    imply that the characteristic polynomial of the resulting endomorphism is $T^r$.
    The Cayley--Hamilton theorem therefore gives $\theta^r=0$.
\end{proof}

Combining this with Theorem \ref{theo: int lift independ to quasi independ},
we obtain the same nilpotence conclusion for any integrally lift-independent
$\frac{1}{p-1}$-Hitchin-small Higgs bundle of rank $r\leqslant g$ on a projective
smooth formal curve over $\calO_{\bfC}$ of genus $g$.
Indeed, its generic fiber has zero $r$-th Higgs power, and the same holds integrally
by $p$-torsion-freeness of the relevant locally free sheaves.

\begin{lemma}\label{lem:rank_of_end_of_nonZero_traceZero_Higgs}
    Let $(\calH,\theta)$ be a Higgs bundle of rank $r$ on $X$. 
    If $\theta \neq 0$ and $\tr(\theta)=0$, then
    $$\rk(\EEnd(\calH,\theta)) \leq (r-1)^2 + 1.$$
    Moreover, if $\tr(\theta^2)=0$ and the equality holds, then $(\calH,\theta)$ is nilpotent of type $(2,1,\cdots,1)$, i.e., $\theta^2=0$ and $\rk(\theta)=1$.
\end{lemma}
\begin{proof}
    Put $F=K(X)$, the field of rational functions on $X$.
    Choose a nonempty open subset $U\subset X$ on which both $\calH$ and $\Omega_X$
    are trivial, and choose a generator $\omega$ of $\Omega_X|_U$. Writing
    $\theta|_U=A\otimes\omega$ and extending scalars to $F$, we obtain
    $A\in\End_F(V)$, where $V=F^r$. Since localization is exact,
    \[
        \rk\bigl(\EEnd(\calH,\theta)\bigr)
        =\dim_F C_F(A),\qquad
        C_F(A):=\{B\in\End_F(V):BA=AB\}.
    \]
    The assumptions imply $A\neq0$ and $\tr(A)=0$. As $F$ has
    characteristic zero, $A$ is not scalar and $r\geq2$.

    Extending scalars to an algebraic closure $\overline F$ does not change
    the dimension of the centralizer, which is the kernel of the linear
    map $[A,-]$. For each eigenvalue $\lambda$ of $A$, let
    $d_{\lambda,j}$ be the number of Jordan blocks for $\lambda$ of size
    at least $j$. The Jordan-block formula for the centralizer gives
    \[
        \dim_F C_F(A)=\sum_{\lambda,j}d_{\lambda,j}^{\,2},\qquad
        \sum_{\lambda,j}d_{\lambda,j}=r.
    \]
    Indeed, two Jordan blocks of sizes $a,b$ for the same eigenvalue
    contribute $\min\{a,b\}$ to the dimension, and blocks with distinct
    eigenvalues contribute zero. Since $A$ is not scalar, at least two
    of the integers $d_{\lambda,j}$ are positive. For positive integers
    with sum $r$ and at least two terms, the sum of their squares is at
    most $(r-1)^2+1$, with equality precisely when the terms are $r-1$
    and $1$. This proves the rank bound.

    Suppose now that equality holds and $\tr(\theta^2)=0$, so that
    $\tr(A^2)=0$. The equality case just described has two possibilities.
    If $A$ has two distinct eigenvalues $\lambda,\mu$, it is semisimple
    with multiplicities $r-1$ and $1$. The trace conditions give
    \[
        (r-1)\lambda+\mu=0,\qquad
        (r-1)\lambda^2+\mu^2=0.
    \]
    Substitution yields $r(r-1)\lambda^2=0$, hence
    $\lambda=\mu=0$, a contradiction. Otherwise, $A$ has a single
    eigenvalue $\lambda$ and Jordan block sizes $(2,1,\ldots,1)$.
    Since $0=\tr(A)=r\lambda$, we have $\lambda=0$. Thus $A^2=0$
    and $\rk(A)=1$ over $\overline F$, and hence also over $F$.

    A morphism between vector bundles on the integral algebraic curve
    $X$ which vanishes over $F$ vanishes identically. Consequently,
    $\theta^2=0$, and the coherent image of $\theta$ has rank $1$.
    This is the meaning of $\rk(\theta)=1$ in the statement, and proves the assertion.
\end{proof}

\subsection{Small rank semistable cohomologically lift-independent Higgs bundles have zero Higgs fields}

Throughout this subsection, let $X$ be a smooth projective connected curve over $K$ of genus $g\geqslant2$.

The main goal of this subsection is to prove the following result.

\begin{theorem}\label{theo: r<(g-1)^{1/2} quasi-independ to 0}
    For any semistable and cohomologically lift-independent Higgs bundle $(\calH,\theta)$ whose rank $r$ satisfies $r\leqslant \sqrt{g} + 1$, $\theta$ must be zero.
\end{theorem}

Combining Theorem \ref{theo: r<(g-1)^{1/2} quasi-independ to 0} with Theorem \ref{theo: int lift independ to quasi independ}, we obtain the following:
\begin{theorem}\label{theo: rk leq sqrt g-1 ss+int independ to 0}
    Let $\frakX$ be a projective smooth formal curve over $\calO_{\bfC}$ of genus $g \geqslant 2$. 
    Let $(\calH, \theta)$ be a $\frac{1}{p-1}$-Hitchin small Higgs bundle on $\frakX$ whose generic fiber is semistable of rank $r\leqslant \sqrt{g} +1 $.
    Then $(\calH, \theta)$ is integrally lift-independent if and only if $\theta=0$.
\end{theorem}

\begin{proof}
    The `if' part is a direct consequence of Lemma \ref{lem:action-of-lift}. 
    For the `only if' part, Theorem \ref{theo: int lift independ to quasi independ} means that $(\calH, \theta)$ is cohomologically lift-independent on the generic fiber of $\frakX$. 
    The claim then follows from Theorem \ref{theo: r<(g-1)^{1/2} quasi-independ to 0}.
\end{proof}

\subsubsection{Globally generated coherent sheaves} 
As a preliminary step, we require a more refined understanding of globally generated vector bundles.

\begin{lemma}\label{lem:Bertini-type}  
    Let $\calE$ be a globally generated vector bundle of rank $r \geqslant 2$ on $X$. 
    Then there exists an injection $\calO_{X} \to \calE$ such that the quotient sheaf is a vector bundle.
\end{lemma}

\begin{proof}
    Consider the incidence variety
    \[ I = \{(x, s) \in X \times \mathrm{H}^0(X,\calE) : s_{x} \in \frakm_{x} \calE_{x} \}. \]
    Since $\calE$ is globally generated, the evaluation map $\rH^0(X, \calE) \to \calE_{x}/\frakm_{x} \calE_{x}$ is surjective. 
    Thus, the dimension of the fiber $I_{x}$ is $\dim \rH^0(X, \calE) - r$. 
    Consequently, $\dim I = \dim X + \dim I_{x} = h^0(X, \calE) - r + 1$. 
    Since $\dim p_2(I) \leqslant \dim I < h^0(X, \calE)$, the set $\rH^0(X, \calE) \setminus p_2(I)$ is non-empty. 
    Choosing $s \in \rH^0(X, \calE) \setminus p_2(I)$ yields an injection $\calO_{X} \to \calE$ with a torsion-free quotient.
\end{proof}

\begin{lemma}\label{lem:globally-generated-bound}
    Let $\calF$ be a globally generated coherent sheaf on $X$. Then
    $$ h^0(X, \calF) \leqslant \deg(\calF) + \rk(\calF).$$
    Moreover, if $\calF$ is also a vector bundle of degree $\leq 2g$ on $X$, we have
    $$ h^0(X, \calF) \leqslant \frac{1}{2}\deg(\calF) + \rk(\calF).$$
\end{lemma}

\begin{proof}
    Considering the canonical short exact sequence $0 \to \calF_{\tor} \to \calF \to \calF_{\operatorname{tf}} \to 0$, we have $h^0(X, \calF) = h^0(X, \calF_{\tor}) + h^0(X, \calF_{\operatorname{tf}})$ and $h^0(X, \calF_{\tor}) = \deg(\calF_{\tor})$. 
    Without loss of generality, we assume $\calF$ is a vector bundle.
    
    If $\calF=\calL$ is a line bundle with $h^0(X, \calL) > 0$, it is well-known that $h^0(X, \calL) \leqslant \deg(\mathcal{L}) + 1$.
    
    For $\rk(\calF) = r \geqslant 2$, Lemma \ref{lem:Bertini-type} provides a short exact sequence $0 \to \calO_{X} \to \calF \to \calF' \to 0$, where $\calF'$ is a globally generated vector bundle. By induction on $r$, we have
    \[ h^0(X, \calF) \leqslant 1 + h^0(X, \calF') \leqslant 1 + \deg(\calF') + (r-1) = \deg(\calF) + \rk(\calF). \]
    
    For the second inequality, the zero bundle is immediate. Let
    $r=\rk(\calF)\geqslant1$ and $d=\deg(\calF)\leqslant2g$.
    Each application of Lemma \ref{lem:Bertini-type} gives a globally
    generated vector bundle quotient of rank one less and of the same
    degree. Iterating, we obtain a globally generated line bundle $\calL$
    of degree $d$ such that
    \[
        h^0(X,\calF)\leqslant r-1+h^0(X,\calL).
    \]
    If $h^1(X,\calL)=0$, Riemann--Roch gives
    \[
        h^0(X,\calL)=d+1-g\leqslant\frac d2+1,
    \]
    where the inequality uses $d\leqslant2g$.
    If $h^1(X,\calL)>0$, global generation implies $h^0(X,\calL)>0$,
    so $\calL\cong\calO_X(D)$ for an effective special divisor $D$.
    Clifford's theorem \cite[IV, Theorem~5.4]{Hartshorne1977} therefore gives
    \[
        h^0(X,\calL)\leqslant\frac12\deg(\calL)+1=\frac d2+1.
    \]
    In either case,
    \[
        h^0(X,\calF)\leqslant r-1+\frac d2+1
        =\frac12\deg(\calF)+\rk(\calF),
    \]
    as required.
\end{proof}

\begin{lemma}\label{lem:centralizer-degree-bound}
    Let $(\calH, \theta)$ be a semistable Higgs bundle on $X$. 
    Then any subsheaf $\calF \subset \EEnd(\calH, \theta)$ satisfies $\deg(\calF) \leqslant 0$.
\end{lemma}

\begin{proof}
    Since $(\calH, \theta)$ is semistable, the dual Higgs bundle $(\calH^\vee, \theta^\vee=-\theta^{\mathrm{t}})$ is also semistable. 
    Consequently, the endomorphism Higgs bundle $\EEnd(\calH) \cong  (\calH \otimes \calH^\vee, \theta \otimes 1 + 1 \otimes \theta^\vee)$ is semistable by \cite[Corollary 3.8]{Simpson1992}.
    The cited result over $\mathbb C$ applies over $K$ by descent to a finitely generated subfield and invariance of semistability under extensions of algebraically closed fields.
    The degree of the bundle $\EEnd(\calH)$ is given by
    $$ \deg(\EEnd(\calH)) = \deg(\calH \otimes \calH^\vee) = \rk(\calH^\vee)\deg(\calH) + \rk(\calH)\deg(\calH^\vee) = 0. $$
    As $\EEnd(\calH,\theta) = \ker([\theta, -])$, any subsheaf $\calF \subset \EEnd(\calH,\theta) $ is a sub-Higgs-sheaf of $\EEnd(\calH)$. 
    The semistability of $\EEnd(\calH)$ then implies
    $$ \deg(\calF) = \rk(\calF)\mu(\calF) \leqslant \rk(\calF)\mu(\EEnd(\calH)) = 0, $$
    where $\mu$ denotes the slope of the bundle.
\end{proof}

Now we are ready to prove Theorem \ref{theo: r<(g-1)^{1/2} quasi-independ to 0}.

\begin{proof}[Proof of Theorem \ref{theo: r<(g-1)^{1/2} quasi-independ to 0}]
    Assume for the sake of contradiction that $\theta: T_{X} \to \EEnd(\calH, \theta)$ is nonzero. 
    By Lemma \ref{lemma: Basic lemma}, $\theta: T_X \to \EEnd(\calH, \theta)$ factors through the trace-zero subbundle 
    $\EEnd(\calH,\theta)^{\tr=0}$ of $\EEnd(\calH,\theta)$. 
    Since $T_X$ is a line bundle and $\EEnd(\calH,\theta)^{\tr=0}$ is torsion-free, we have the short exact sequence
    $$ 0 \to T_X \xrightarrow[]{\theta} \EEnd(\calH,\theta)^{\tr=0} \to \calQ \to 0. $$
    This induces the long exact sequence
    $$0 \xrightarrow[]{} \rH^0(X,T_{X}) \xrightarrow[]{} \rH^0(X,\EEnd(\calH,\theta)^{\tr=0}) \xrightarrow[]{}\rH^0(X, \calQ) \xrightarrow[]{\partial} \rH^1(X, T_X) \xrightarrow[]{} 0$$
    in cohomology by the cohomological lift-independence condition $\rH^1(\theta)=0$. Consequently,
    \begin{equation}\label{ineq:long_exact_sequence}
         h^0(X, \calQ) = h^0(X,\EEnd(\calH,\theta)^{\tr=0}) - \chi(X, T_X) = h^0(X,\EEnd(\calH,\theta)^{\tr=0}) + 3g-3\geq 3g - 3.
    \end{equation}
    Let $\calQ' = \im(\rH^0(X, \calQ) \otimes \calO_{X} \to \calQ)$. Then $\calQ'$ is globally generated and $h^0(X, \calQ') = h^0(X, \calQ)$.
    Let $\calF = \calQ' \times_{\calQ} \EEnd(\calH, \theta)^{\tr=0} \subset \EEnd(\calH, \theta)^{\tr=0}$. 
    By Lemma \ref{lem:centralizer-degree-bound}, we have $\deg(\calF) \leqslant 0$. 
    From the short exact sequence
    $$ 0 \to T_X \to \calF \to \calQ' \to 0, $$
    we obtain
    \begin{equation}\label{ineq:degree}
        \deg(\calQ') = \deg(\calF) - \deg(T_X) \leqslant 0 - (2 - 2g) = 2g - 2.
    \end{equation}
    The rank also satisfies 
    \begin{equation}\label{ineq:rank}
        \rk(\calQ') = \rk(\calF) - 1 \leqslant \rk(\EEnd(\calH, \theta)^{\tr=0}) - 1 \leqslant (r-1)^2 - 1.
    \end{equation}
    Applying Lemma \ref{lem:globally-generated-bound}, we have
    \begin{equation}\label{ineq:deg-rank}
        h^0(X, \calQ') \leqslant \deg(\calQ') + \rk(\calQ') \leqslant (2g - 2) + ((r-1)^2 - 1) = 2g + (r-1)^2 - 3. 
    \end{equation}
    Thus,
    $$ 3g - 3 \leqslant h^0(X, \mathcal{Q}) = h^0(X, \mathcal{Q}') \leqslant 2g + (r-1)^2 - 3,$$
    which implies $g \leqslant (r-1)^2$. This forces $g = (r-1)^2$.

    Assume that $g = (r-1)^2$ and $\theta$ is nonzero. 
    In this case, the inequalities (\ref{ineq:long_exact_sequence}), (\ref{ineq:degree}), (\ref{ineq:rank}) and (\ref{ineq:deg-rank}) are equalities.
    In particular, the following conditions hold: 
    \begin{enumerate}
        \item[(a)] $\deg(\calF)=0,\deg(\calQ')= 2g - 2$.
        \item[(b)] $\rk(\calF)=\rk (\EEnd(\calH,\theta)^{\tr=0})=(r-1)^2$.
        \item[(c)] $h^0(X, \calQ') = \deg(\calQ') + \operatorname{rk}(\calQ')$.
        \item[(d)] $h^0(X,\EEnd(\calH,\theta)^{\tr=0})=0$.
    \end{enumerate}
    
    Conditions (a) and (b), together with Lemma \ref{lem:centralizer-degree-bound}, imply that $\calF = \EEnd(\calH,\theta)^{\tr=0}$, so $\calQ =\calQ'$ is globally generated. 
    By Lemma \ref{lem:globally-generated-bound}, we deduce from conditions (a) and (c) that $\deg (\calQ/\calQ_{\tor})=0$. 
    We conclude that $\calQ/\calQ_{\tor} \cong \calO^{\oplus((r-1)^2-1)}_{X}$. 
    Let $\calL$ be the saturation of $T_{X}$ in $\EEnd(\calH,\theta)^{\tr=0}$. We obtain a short exact sequence 
    $$ 0 \to \calL \xrightarrow[]{\mathrm{incl}} \EEnd(\calH,\theta)^{\tr=0} \to \calQ/\calQ_{\tor}=  \calO^{\oplus((r-1)^2-1)}_{X}\to 0.$$
    Combining this with condition (d), we deduce that $\calL$ is a nontrivial line bundle of degree $0$. 
    In particular, the Jordan--H\"older filtration of $\EEnd(\calH,\theta)^{\tr=0}$ has the associated graded object 
    $$\gr(\EEnd(\calH,\theta)^{\tr=0})=\calL \oplus \calO^{\oplus((r-1)^2-1)}_{X}.$$

    By Lemma \ref{lemma: Basic lemma}, we have $\tr(\theta)=\tr(\theta^2)=0$. 
    Applying Lemma \ref{lem:rank_of_end_of_nonZero_traceZero_Higgs} to condition (b), we deduce that $(\calH,\theta)$ must be nilpotent of type $(2,1,\cdots,1)$. 
    The inclusion $\calL\subset \EEnd(\calH,\theta)$ induces a morphism $N: \calH \to \calH \otimes \calL^{\otimes (-1)}$ commuting with $\theta$. 
    Consider the filtration 
    $$
    0 \subset \im(N)\otimes \calL \subset \ker(N) \subset \calH.
    $$
    Let $\calI$ be the saturation of $\im(N)$ in $\calH \otimes \calL^{\otimes (-1)}$. The semistability of $(\calH,\theta)$ implies that 
    $$ \mu(\calH) \leq \mu(\im(N)) \leq \mu(\calI) \leq \mu(\calH\otimes \calL^{\otimes (-1)} )= \mu(\calH).$$
    Thus $\im(N)=\calI$, i.e., $\im(N)$ is saturated. Hence $\ker(N)/(\im(N) \otimes \calL)$ is a vector bundle. 
    
    The filtration above induces a filtration $F$ on $\EEnd(\calH,\theta)^{\tr=0}$ whose associated graded object is 
    $$
    \begin{aligned}
        \gr_{F}(\EEnd(\calH,\theta)^{\tr=0}) \cong & \EEnd(\ker(N)/(\im(N) \otimes \calL)) \oplus \HHom(\calH/\ker(N),\ker(N)/(\im(N) \otimes \calL)) \\
        &\oplus \HHom(\ker(N)/(\im(N) \otimes \calL),\im(N)\otimes \calL) \oplus \calL .
    \end{aligned}
    $$
    By considering the Jordan--H\"older filtration, we deduce that 
    $$(\calH/\ker(N))^{\oplus(r-2)}\cong \gr_{\mathrm{JH}}(\ker(N)/(\im(N) \otimes \calL)) \cong (\im(N)\otimes \calL)^{\oplus(r-2)}.$$
    Since $r-2>0$ and both $\calH/\ker(N)$ and
$\im(N)\otimes\calL$ are line bundles of the same degree,
the uniqueness of Jordan--H\"older factors gives
$\calH/\ker(N)\cong\im(N)\otimes\calL$.
    But $\calH/\ker(N) \cong \im(N)$, so $\calL \cong \calO_{X}$, contradicting the fact that $\calL$ is nontrivial.
\end{proof}

\section{Lift-Independent Higgs Bundles with Nonzero Higgs Fields}\label{section: non zero independ}

\subsection{Lift-independent Higgs bundles with nilpotent but nonzero Higgs fields}
The main goal of this subsection is to prove the following result. 

\begin{theorem}\label{theo: nonzero-nilp-deg0-ss-quasi-indep-Higgs}
    Let $K$ be an algebraically closed field of characteristic $0$, and let $X$ be a smooth projective curve of genus $g \geqslant 2$ over $K$. 
    Then, there exists a semistable, non-polystable, nilpotent Higgs bundle $(\mathcal{H}, \theta)$ of rank $g+1$ and degree $0$ with nonzero Higgs field on $X$ that is cohomologically lift-independent.
\end{theorem}

Before introducing the construction, we need some lemmas.

\begin{lemma}\label{lem:ext-deg-zero-semistable}
    Let $X$ be a projective smooth curve of genus $g \geqslant 2$. 
    Let $\calL$ be a non-trivial line bundle of degree $0$ on $X$. 
    There exists a short exact sequence
    \begin{equation*}
        0 \xrightarrow[]{} \calL \xrightarrow[]{i} \calV \xrightarrow[]{} \calO_{X}^{\oplus (g-1)} \xrightarrow[]{} 0,
    \end{equation*}
    where $\calV$ is a semistable vector bundle of degree $0$, such that the connecting homomorphism
    $$
    \rH^0(X, \calO_{X}^{\oplus (g-1)}) \xrightarrow[]{\delta} \rH^1(X,\calL)
    $$ 
    is an isomorphism. 
    Consequently, the induced map $\rH^1(X, \calL) \xrightarrow[]{\rH^1(i)} \rH^1(X, \calV)$ is zero.
\end{lemma}
\begin{proof}
    Since $\calL,\calL^{-1}$ are non-trivial line bundles of degree $0$, we have $h^0(X,\calL)=h^0(X,\calL^{-1})=0$. 
    Applying the Riemann--Roch theorem yields 
    $$h^1(X,\calL)=h^0(X,\calL\otimes \Omega_{X})=g-1.$$
    Since $\Ext^1(\calO_{X}^{\oplus (g-1)}, \calL) \cong  \Hom(\rH^0(X, \calO_{X}^{\oplus (g-1)}), \rH^1(X, \calL))$, we may choose an extension class that induces an isomorphism $\rH^0(X, \mathcal{O}_{X}^{\oplus (g-1)}) \xrightarrow[]{\sim} \rH^1(X, \calL)$. 
    Let  \begin{equation}\label{eq:ext-deg-zero-semistable}
        0 \xrightarrow[]{} \calL \xrightarrow[]{i} \calV \xrightarrow[]{} \calO_{X}^{\oplus (g-1)} \xrightarrow[]{} 0
    \end{equation} be the corresponding extension. 

    From the long exact sequence associated with the short exact sequence \eqref{eq:ext-deg-zero-semistable}, 
    the induced map $\rH^1(i):\rH^1(X, \calL) \to \rH^1(X, \calV)$ is zero. 
    
    Finally, since both $\calL$ and $\calO_{X}^{\oplus (g-1)}$ are semistable vector bundles of degree $0$, so is their extension $\calV$.
\end{proof}

\begin{lemma}\label{lem:nonzero-nilp-deg0-ss-quasi-indep-Higgs}
    Let $\calL,\calV$ and $\calL \xrightarrow[]{i} \calV$ be defined as in Lemma \ref{lem:ext-deg-zero-semistable}. 
    Choose a non-zero section $\omega \in \rH^0(X,\calL \otimes \Omega_{X})$. 
    Let $\calH = \calV \oplus \calO_X$. 
    Define a Higgs field $\theta: \calH \to \calH \otimes \Omega_{X}$ by
    $$ 
    \calH \xrightarrow[]{\operatorname{pr}} \mathcal{O}_X \xrightarrow[]{\omega} \calL \otimes \Omega_X \xrightarrow[]{i \otimes 1} \mathcal{V} \otimes \Omega_X \xrightarrow[]{\operatorname{incl}} \calH \otimes \Omega_X. 
    $$
    Then $(\mathcal{H}, \theta)$ is a semistable, non-polystable, nilpotent Higgs bundle of degree $0$ with $\theta \neq 0$.
\end{lemma}

\begin{proof}
    In the proof of Lemma \ref{lem:ext-deg-zero-semistable}, we have shown that 
    $$
    h^0(X,\calL \otimes \Omega_{X}) = g-1 >0, 
    $$
    so it is possible to choose a non-zero section $\omega \in \rH^0(X,\calL \otimes \Omega_{X})$. 
    By construction, it is evident that $\theta \neq 0$ and $\theta^2 = 0$. 
    
    As both $\calO_{X}$ and $\calV$ are semistable vector bundles of degree $0$, so is their direct sum $\calH$. 
    In particular, $(\calH,\theta)$ is a semistable Higgs bundle. 

    Finally, since we have an exact sequence 
    $$
    0 \xrightarrow[]{} (\calV, 0) \xrightarrow[]{} (\calH,\theta) \xrightarrow[]{} (\calO_{X}, 0) \xrightarrow[]{} 0 
    $$
    of Higgs bundles with $\theta \neq 0$, the sub-Higgs-bundle $(\calV, 0)$ of $(\calH,\theta)$ is not a direct summand. 
    We conclude that the Higgs bundle $(\calH,\theta)$ is not polystable.
\end{proof}

\begin{proof}[Proof of Theorem \ref{theo: nonzero-nilp-deg0-ss-quasi-indep-Higgs}]
    We claim that the Higgs bundle $(\mathcal{H},\theta)$ constructed in Lemma \ref{lem:nonzero-nilp-deg0-ss-quasi-indep-Higgs} is cohomologically lift-independent.

    Consider the direct summand $\calV \cong \HHom(\calO_X, \calV) \subset \EEnd(\calH)$. 
    Note that $\calV \subset \EEnd(\calH, \theta)$ since $\theta|_{\calV} = 0$ and $\im(\theta) \subset \calV \otimes \Omega_X$. 
    By the construction of $\theta$, there is a factorization
    \[ \theta: T_X \xrightarrow[]{\omega} \calL \xrightarrow[]{i} \calV \hookrightarrow \EEnd(\mathcal{H}, \theta), \]
    which induces the factorization of maps on cohomology:
    \[ \rH^1(\theta) : \rH^1(X, T_X) \xrightarrow[]{\rH^1(\omega)} \rH^1(X, \calL) \xrightarrow{\rH^1(i)} \rH^1(X, \calV) \to \rH^1(X, \EEnd(\mathcal{H}, \theta)). \]
    Since $\rH^1(i) = 0$ by Lemma \ref{lem:ext-deg-zero-semistable}, we conclude that $\rH^1(\theta) = 0$, i.e., $(\mathcal{H}, \theta)$ is cohomologically lift-independent.
\end{proof}

Indeed, the above example is also lift-independent. 

\begin{theorem}\label{theo: nonzero-nilp-deg0-ss-rat-indep-Higgs}
    Assume that $K=\bfC$. Then the Higgs bundle $(\calH, \theta)$ constructed in Lemma \ref{lem:nonzero-nilp-deg0-ss-quasi-indep-Higgs} is lift-independent.
\end{theorem}

\begin{proof}
    Fix an affinoid cover $\frakU = \{U_1, U_2\}$ of $X$. 
    For any \v{C}ech cocycle $a \in \check{Z}^1(\frakU, T_X)$, we know that $\theta(a) \in \check{B}^1(\frakU, \EEnd(\calH, \theta))$. 
    That is, there exist local sections $\psi_j = \psi_j(a) \in \Gamma(U_j, \EEnd(\calH, \theta))$ for $j=1, 2$, such that 
    $$ \theta(a) = \psi_2|_{U_{12}} - \psi_1|_{U_{12}}, $$
    where $U_{12}=U_1\cap U_2$.
    By the proof of Theorem \ref{theo: nonzero-nilp-deg0-ss-quasi-indep-Higgs}, the induced map on cohomology $\rH^1(i\circ\omega): \rH^1(X, T_X) \to \mathrm{H}^1(X, \calV)$ is zero. 
    Thus, we may refine our choice of $\psi_j$ such that $\psi_j^2=0$ for all $j$.

    Let $\varphi_j = \exp(\psi_j) = \operatorname{id} + \psi_j$. 
    Since $\psi_1$ and $\psi_2$ take values in $\mathcal{V}$ and square to zero, they commute with each other. We then obtain the following commutative diagram on the overlap $U_{12}$:
    \[
    \begin{tikzcd}
        \mathcal{H}|_{U_{12}} \arrow[r, "\operatorname{id}"] \arrow[d, "\varphi_1"'] & \mathcal{H}|_{U_{12}} \arrow[d, "\varphi_2"] \\
        \mathcal{H}|_{U_{12}} \arrow[r, "\exp(\theta(a))"'] & \mathcal{H}|_{U_{12}}
    \end{tikzcd}
    \]
    which implies that the Higgs bundle obtained by twisting the descent datum by $\exp(\theta(a))$ is isomorphic to $(\calH,\theta)$. 
    Therefore, $(\calH, \theta)$ is lift-independent by Remark \ref{rmk:rat-transition}.
\end{proof}

\subsection{Non-nilpotent integrally lift-independent Higgs bundles} 
We will prove the following result.

\begin{theorem}\label{theo: non-nilpotent cohomologically lift-independent semistable}
    There exists a projective smooth connected curve $X/\mathbb{C}$ of genus $2$, 
    and a non-nilpotent, cohomologically lift-independent semistable Higgs bundle $(\mathcal{H},\theta)$ of degree $0$ on $X$. 
\end{theorem}

The construction relies on a solution to Prill's problem in \cite{LandesmanLitt2024}.

\begin{lemma}\label{lem:Prill-problem}
    There exists a projective smooth connected curve $X/\bC$ of genus $2$, 
    a finite \'etale connected cover $\pi : Y\to X$, 
    and a nonzero section $\lambda \in \mathrm{H}^0(Y,\pi^* \Omega_{X})=\mathrm{H}^0(Y,\Omega_{Y})$, 
    such that the trace-product map 
    \begin{align*}
        \mathrm{H}^0(Y,\Omega_{Y}) & \xrightarrow[]{} \mathrm{H}^0(X,\Omega_{X}^{\otimes 2}) \\
        \alpha & \mapsto \tr_{Y/X}(\lambda \alpha)
    \end{align*}
    is zero.
\end{lemma}
\begin{proof}
    The construction is in \cite[Proposition 2.4]{LandesmanLitt2024}. 
    The proof of the required properties is in the proof of \cite[Proposition 2.3]{LandesmanLitt2024}.
\end{proof}

\begin{lemma}\label{lem:semi-stability-of-fet-pushforward-of-structure-sheaf}
    Let $\pi : Y \to X$ be a finite \'etale map of projective smooth connected curves over an algebraically closed field $K$.
    Then $\pi_{*}\calO _{Y}$ is a semistable vector bundle of degree $0$ on $X$. 
\end{lemma}
\begin{proof}
    By \cite[\href{https://stacks.math.columbia.edu/tag/0AYZ}{Tag 0AYZ}]{stacks-project}, $\deg(\pi_{*}\calO _{Y})=0$.
    
    It remains to show the semistability of $\pi_{*} \calO _{Y}$. 
    Let $d=\deg(\pi)$. 
    Let $r : Z\to X$ be a Galois closure of $\pi$. Then 
    \[r^*\pi_{*} \calO _{Y} \cong  \calO _{Z}^{d},\]
    because $Y\times_{X}Z$ is a disjoint union of $d$ copies of $Z$. 
    If $F\subset \pi_{*}\calO _{Y}$ is a subbundle of positive degree, 
    then $r^*F\subset r^*\pi_{*} \calO _{Y} \cong  \calO _{Z}^{d}$ is also a subbundle of positive degree. 
    But the trivial bundle $ \calO _{Z}^{d}$ is semistable, so it has no positive degree subbundle. 
    Therefore every subbundle of $\pi_{*}\calO _{Y}$ has degree at most $0$, and hence $\pi_{*}\calO _{Y}$ is semistable.
\end{proof}

\begin{lemma}\label{lem:non-nilp-quasi-indep-Higgs}
    Let $X$ be a projective smooth connected curve over an algebraically closed field $K$,
    let $\pi : Y \to X$ be a finite \'etale connected cover, 
    and let $\lambda \in \mathrm{H}^0(Y,\pi^* \Omega_{X})=\mathrm{H}^0(Y,\Omega_{Y})$ be a nonzero section. 
    Put $\mathcal{H}=\pi_{*}\calO _{Y}$ and let 
    \[ \theta : \pi_{*}\calO _{Y} \xrightarrow[]{\pi_{*}\lambda} \pi_{*}\Omega_{Y} \cong  \pi_{*}\pi^*\Omega_{X} \cong  \pi_{*}\calO _{Y} \otimes \Omega_{X}.\] 
    Then $(\mathcal{H},\theta)$ is a non-nilpotent semistable Higgs bundle of degree $0$ on $X$.  
    Moreover, if the trace product map 
    \[ \mathrm{H}^0(Y,\Omega_{Y}) \xrightarrow[]{\tr_{Y/X}(\lambda - )} \mathrm{H}^0(X,\Omega_{X}^{\otimes 2}) \]
    is zero, then $(\mathcal{H},\theta)$ is cohomologically lift-independent. 
\end{lemma}
\begin{proof}
    By Lemma \ref{lem:semi-stability-of-fet-pushforward-of-structure-sheaf}, the vector bundle $\mathcal{H}$ is semistable of degree $0$. 
    Hence the Higgs bundle $(\mathcal{H},\theta)$ is semistable of degree $0$.    
    Since $\lambda \ne 0$, there exists a non-empty affine open subset $U$ of $X$, a local coordinate $t$ of $U$, and a nonzero section $a \in \mathrm{H}^0(\pi^{-1}(U),\calO _{Y})=\mathrm{H}^0(U,\pi_{*}\calO _{Y})$, such that 
    \[ \lambda = a \pi^* {\rm d}t.\]
    Let $\partial$ be the dual basis of ${\rm d}t$. 
    Then $\theta(\partial)$ is multiplication by $a$ on $\pi_{*}\calO _{Y}|_{U}$. As $Y$ is a reduced scheme, $a$ is a non-nilpotent element of $\mathrm{H}^0(\pi^{-1}(U),\calO _{Y})$, so multiplication by $a$ is not nilpotent. 
    Thus the Higgs field $\theta$ is not nilpotent.   
    The $\calO _{X}$-algebra $\pi_{*}\calO _{Y}$ acts on itself by multiplication, which induces an injection 
    \[ i : \pi_{*}\calO _{Y} \xrightarrow[]{} \EEnd(\mathcal{H},\theta),\] 
    because multiplication operators commute with multiplication by $\lambda$ and $Y$ is an integral scheme. 
    The canonical map $\theta : T_{X} \to \EEnd(\mathcal{H},\theta)$ factors as 
    \[ \theta : T_{X} \xrightarrow[]{u_{\lambda}} \pi_{*}\calO _{Y} \xrightarrow[]{i} \EEnd(\mathcal{H},\theta),\]
    where $u_{\lambda}$ is obtained by contracting the section $\lambda$ with vector fields on $X$.  
    By Serre duality, the dual of 
    \[ \mathrm{H}^1(X,T_{X}) \xrightarrow[]{\mathrm{H}^1(u_{\lambda})} \mathrm{H}^1(Y,\calO _{Y})\cong  \mathrm{H}^1(X,\pi_*\mathcal{O}_Y)\] 
    is identified with the trace-product map 
    \[\mathrm{H}^0(Y,\Omega_{Y}) \xrightarrow[]{\tr_{Y/X}(\lambda - )} \mathrm{H}^0(X,\Omega_{X}^{\otimes 2})\]
    under the identification $\pi_*\mathcal{O}_Y\cong \HHom_{\mathcal{O}_X}(\pi_*\mathcal{O}_Y,\mathcal{O}_X)$ induced by $\operatorname{tr}_{Y/X}$.
    If the trace-product map is zero, then $\mathrm{H}^1(u_{\lambda})=0$, so $\mathrm{H}^1(\theta)=0$. 
    Thus $(\mathcal{H},\theta)$ is cohomologically lift-independent. 
\end{proof}

\begin{proof}[Proof of Theorem \ref{theo: non-nilpotent cohomologically lift-independent semistable}]
    By Lemma \ref{lem:Prill-problem}, there exists a projective smooth connected curve $X/\mathbb{C}$ of genus $2$, 
    a finite \'etale connected cover $\pi : Y\to X$, 
    and a nonzero section $\lambda \in \mathrm{H}^0(Y,\pi^* \Omega_{X})=\mathrm{H}^0(Y,\Omega_{Y})$, 
    such that the trace-product map 
      \[\mathrm{H}^0(Y,\Omega_{Y}) \xrightarrow[]{\tr_{Y/X}(\lambda - )} \mathrm{H}^0(X,\Omega_{X}^{\otimes 2})\]
    is zero. By Lemma \ref{lem:non-nilp-quasi-indep-Higgs}, $(\pi_{*}\calO _{Y},\pi_{*}\lambda)$ is a non-nilpotent cohomologically lift-independent semistable Higgs bundle on $X$. 
\end{proof}

\begin{lemma}\label{lem:spread-nonnilpotent}
Let $X/\bC$, $f:Y\to X$, and $\lambda\in \rH^0(Y,\Omega_{Y})=\rH^0(Y,f^*\Omega_X)$ be as in Lemma \ref{lem:Prill-problem}. Put
\[
    H:=f_*\calO_{Y},
\]
viewed as a commutative $\calO_{X}$-algebra, and let $\theta$ be the Higgs field on $H$ given by multiplication by $\lambda$, i.e.
\[
        \theta:H \xrightarrow[]{\ \lambda\ } H \otimes \Omega_{X}.
\]
After replacing the data by isomorphic data, there exist a finitely generated $\bZ$-subalgebra $R\subset \bC$, a smooth projective morphism
\[
        \pi:\calX \longrightarrow \Spec R
\]
with geometrically connected fibers of genus $2$, a finite \'etale morphism
\[
        f_{R}:\calY \longrightarrow \calX,
\]
and a section
\[
        \lambda_{R}\in \rH^0(\calY,\Omega_{\calY/R})
        =\rH^0(\calY,f_R^*\Omega_{\calX/R}),
\]
such that the following assertions hold.

\begin{enumerate}
\item After base change along $\Spec \bC\to \Spec R$, the tuple
\[
        (\calX,\calY,f_{R},\lambda_{R})
\]
is identified with $(X,Y,f,\lambda)$.

\item Let
\[
        \calH_{R}:=f_{R,*}\calO_{\calY},
\]
viewed as a commutative $\calO_{\calX}$-algebra, and let
\[
        \theta_{R}:\calH_{R}\longrightarrow
        \calH_{R}\otimes\Omega_{\calX/R}
\]
be the relative Higgs field defined by multiplication by $\lambda_{R}$. For every geometric point $\bar s\to \Spec R$, the geometric fiber
\[
        (\calH_{\bar s},\theta_{\bar s})
\]
is a non-nilpotent semistable Higgs bundle of degree $0$ on the smooth projective curve $\calX_{\bar s}$ of genus $2$.

\item For every geometric point $\bar s\to \Spec R$, the map
\[
        \rH^0(\calY_{\bar s},\Omega_{\calY_{\bar s}})
        \longrightarrow
        \rH^0(\calX_{\bar s},\Omega_{\calX_{\bar s}}^{\otimes 2}),
        \qquad
        \alpha\longmapsto
        \operatorname{Tr}_{\calY_{\bar s}/\calX_{\bar s}}
        (\lambda_{\bar s}\alpha)
\]
is zero. Consequently, $(\calH_{\bar s},\theta_{\bar s})$ is cohomologically lift-independent.
\end{enumerate}
\end{lemma}

\begin{proof}
We first spread out the data. 
Since $\mathbb C$ is the filtered colimit of its finitely generated $\mathbb Z$-subalgebras, the projective smooth curve $X/\mathbb C$, the finite \'etale morphism $f:Y\to X$, and the section $\lambda\in H^0(Y,\Omega_Y)$ all descend, by standard limit arguments, to some finitely generated $\mathbb Z$-subalgebra $R_0\subset \mathbb C$.
Thus, after replacing $R_0$ by a localization, we obtain a smooth projective morphism
\[
        \pi_0:\calX_0\to \Spec R_0,
\]
a finite \'etale morphism
\[
        f_0:\calY_0\to \calX_0,
\]
and a section
\[
        \lambda_0\in \rH^0(\calY_0,\Omega_{\calY_0/R_0}),
\]
whose base change to $\bC$ recovers the original data $(X,Y,f,\lambda)$.

We now shrink $\Spec R_0$ further in order to make the desired properties hold on every geometric fiber. 
Since $R_0\subset \bC$, it is an integral domain. 
The number of geometric connected components of the fibers of a smooth proper morphism is locally constant, and the genus of the fibers of a smooth proper curve is locally constant. 
Hence, after localizing $R_0$, we may assume that all geometric fibers of $\calX_0\to \Spec R_0$ are connected smooth projective curves of genus $2$.

Set
\[
        \calH_{0}:=f_{0,*}\calO_{\calY_{0}}.
\]
Since $f_{0}$ is finite \'etale, $\calH_{0}$ is a finite locally free commutative $\calO_{\calX_{0}}$-algebra. 
The section $\lambda_{0}$ defines a relative Higgs field
\[
        \theta_{0}:\calH_{0}\longrightarrow \calH_{0}\otimes\Omega_{\calX_{0}/R_{0}}.
\]
The condition that the fiberwise Higgs field be non-nilpotent is open: 
indeed, if $r=\rk(\calH_0)$, then nilpotency is detected by the vanishing of $\theta_{0}^r$, and the vanishing locus of this morphism is closed on the base (as $\pi_{0}$ is smooth). 
Since the original Higgs field over $\bC$ is non-nilpotent, after shrinking $R_{0}$, we may assume that $\theta_{\bar s}$ is non-nilpotent for every geometric point $\bar s\to \Spec R_{0}$.

For every geometric point $\bar s\to \Spec R_{0}$, the bundle
\[
        \calH_{\bar s}
        =f_{\bar s,*}\calO_{\calY_{\bar s}}
\]
has degree $0$ and is semistable. 
Indeed, after pulling back along a Galois closure of $f_{\bar s}$, the bundle $\calH_{\bar s}$ becomes a trivial bundle; hence it cannot have a subbundle of positive degree. 
Thus the Higgs bundle $(\mathcal H_{\bar s},\theta_{\bar s})$ is semistable of degree $0$.

It remains to spread out the cohomological lift-independence. 
Since $f_0$ is finite \'etale, we have
\[
        \Omega_{\mathcal Y_0/R_0}\simeq f_0^*\Omega_{\mathcal X_0/R_0}.
\]
The section $\lambda_0$ gives a trace-product morphism of coherent $R_0$-modules
\[
        \tau_{\lambda_0}:
        \rH^0(\mathcal Y_0,\Omega_{\mathcal Y_0/R_0})
        \longrightarrow
        \rH^0(\mathcal X_0,\Omega_{\mathcal X_0/R_0}^{\otimes 2}),
        \qquad
        \alpha\longmapsto
        \operatorname{Tr}_{\mathcal Y_0/\mathcal X_0}(\lambda_0\alpha).
\]
After further localization, cohomology and base change hold for the two vector bundles appearing above, so $\tau_{\lambda_0}$ is a morphism between finite locally free $R_0$-modules. 
Its base change to $\mathbb C$ is zero by the choice of the original data. 
Since $R_0$ is an integral domain, a morphism of finite $R_0$-modules which vanishes at the generic point becomes zero after localizing $R_0$ once more. 
Hence we may assume that
\[
        \tau_{\lambda_0}=0
\]
over $\Spec R_0$. 
By base change, the trace-product map is therefore zero on every geometric fiber.

Finally, by Lemma \ref{lem:non-nilp-quasi-indep-Higgs}, the vanishing of the trace-product map is Serre-dual to the vanishing of the map
\[
        \rH^1(\mathcal X_{\bar s},T_{\mathcal X_{\bar s}})
        \longrightarrow 
        \rH^1(\mathcal X_{\bar s},\mathcal H_{\bar s})
        \hookrightarrow
        \rH^1(\mathcal X_{\bar s},
        \EEnd(\mathcal H_{\bar s},
        \theta_{\bar s})).
\]
Thus every fiber $(\mathcal H_{\bar s},\theta_{\bar s})$ is cohomologically lift-independent. 
Taking $R$ to be the final localization of $R_0$ gives the lemma.
\end{proof}

\begin{theorem}\label{thm:integral-nonnilpotent-example}
For all sufficiently large primes $p$, there exists a complete algebraically
closed extension $C/\mathbb Q_p$, a smooth projective formal curve
$\mathfrak X$ over $\mathcal O_C$ of genus $2$, and a non-nilpotent
$\frac{1}{p-1}$-Hitchin-small Higgs bundle $(\mathcal H,\theta)$ on
$\mathfrak X$, whose generic fiber is semistable of degree $0$, such that
$(\mathcal H,\theta)$ is integrally lift-independent.
\end{theorem}

\begin{proof}
Let
\[
        (\mathcal X,\mathcal Y,f_R,\lambda_R)
\]
be the model over the finitely generated $\mathbb Z$-algebra $R$ constructed
in Lemma~\ref{lem:spread-nonnilpotent}. We first choose a $p$-adic
specialization of $R$. Since $R$ is finitely generated over $\mathbb Z$,
there exists a closed point of $\Spec(R\otimes_{\mathbb Z}\mathbb Q)$, with
residue field a number field $L$. After excluding finitely many primes, the
images of all elements of $R$ are integral at every prime of $L$ above
$p$. Hence, for every sufficiently large prime $p$, after choosing an
embedding $L\hookrightarrow \overline{\mathbb Q}_p$, we obtain a homomorphism
\[
        R\longrightarrow \mathcal O_C
\]
for some complete algebraically closed extension $C/\mathbb Q_p$.

Base change the whole family along $R\to \mathcal O_C$. Let
\[
        \mathcal X_{\mathcal O_C}
        :=\mathcal X\times_{\Spec R}\Spec\mathcal O_C,
        \qquad
        \mathcal Y_{\mathcal O_C}
        :=\mathcal Y\times_{\Spec R}\Spec\mathcal O_C.
\]
Let
\[
        f:\mathcal Y_{\mathcal O_C}\to \mathcal X_{\mathcal O_C}
\]
be the induced finite \'etale morphism, and put
\[
        \mathcal H_{\mathcal O_C}
        :=f_*\mathcal O_{\mathcal Y_{\mathcal O_C}},
\]
viewed as a commutative $\mathcal O_{\mathcal X_{\mathcal O_C}}$-algebra.
The section $\lambda_R$ gives, after base change, a section
\[
        \lambda\in \rH^0(\mathcal Y_{\mathcal O_C},
        \Omega_{\mathcal Y_{\mathcal O_C}/\mathcal O_C})
        =\rH^0(\mathcal Y_{\mathcal O_C},
        f^*\Omega_{\mathcal X_{\mathcal O_C}/\mathcal O_C}).
\]
Multiplication by $\lambda$ defines an ordinary relative Higgs field
\[
        \theta' : \mathcal H_{\mathcal O_C}\longrightarrow
        \mathcal H_{\mathcal O_C}\otimes
        \Omega_{\mathcal X_{\mathcal O_C}/\mathcal O_C}.
\]

Let $\mathfrak X$ be the $p$-adic completion of
$\mathcal X_{\mathcal O_C}$, and let $\mathcal H$ be the $p$-adic
completion of $\mathcal H_{\mathcal O_C}$. It is a commutative
$\mathcal O_{\mathfrak X}$-algebra as well as the underlying vector bundle
of the resulting Higgs bundle. The generic fiber of $\mathfrak X$ is a
smooth projective curve of genus $2$. By
Lemma~\ref{lem:spread-nonnilpotent}, the generic fiber of
$(\mathcal H,\theta')$ is semistable of degree $0$, and the Higgs field
$\theta'$ is non-nilpotent.

We now scale the Higgs field. Choose once and for all a trivialization of the Tate twist. Since
\[
        v_p(p)=1>\frac{1}{p-1}
\]
for $p\ge 3$, the scaled Higgs field
\[
        \theta:=p\theta'
\]
takes values in the integral Higgs sheaf
\[
        \mathcal H\otimes \rho \widehat{\Omega}^1_{\mathfrak X/\mathcal O_C}(-1),
        \qquad \rho=\zeta_p-1,
\]
and is $\frac{1}{p-1}$-Hitchin-small. Multiplication by the nonzero scalar
$p$ does not change nilpotency on the generic fiber, so $\theta$ is still
non-nilpotent generically.

It remains to prove integral lift-independence. Let
\[
        \widetilde{\mathfrak X}_1,\widetilde{\mathfrak X}_2
\]
be two flat $A_2$-liftings of $\mathfrak X$. Choose a small affine open
cover $\mathfrak U=\{U_i\}$ of $\mathfrak X$. By the usual deformation
theory of $A_2$-liftings, the difference class
\[
        [\widetilde{\mathfrak X}_2]-[\widetilde{\mathfrak X}_1]
\]
is represented by a Čech $1$-cocycle
\[
        (a_{ij})\in
        \check Z^1(\mathfrak U,\rho^{-1}T_{\mathfrak X/\mathcal O_C}(1)).
\]
After the chosen trivialization of the Tate twist, write
\[
        (b_{ij}):= (\rho a_{ij})\in
        \check Z^1(\mathfrak U,T_{\mathfrak X/\mathcal O_C}).
\]

The Higgs field $\theta'$ factors through the commutative
$\mathcal O_{\mathfrak X}$-algebra $\mathcal H$:
\[
        T_{\mathfrak X/\mathcal O_C}
        \xrightarrow{\,u_\lambda\,}
        \mathcal H
        \hookrightarrow
        \EEnd(\mathcal H,\theta'),
\]
where $u_\lambda$ is contraction with $\lambda$. By the trace-product
vanishing in Lemma~\ref{lem:spread-nonnilpotent}, equivalently by Serre duality,
the induced map
\[
        \rH^1(\mathfrak X,T_{\mathfrak X/\mathcal O_C})
        \xrightarrow{\,H^1(u_\lambda)\,}
        \rH^1(\mathfrak X,\mathcal H)
\]
is zero. Hence the cocycle $u_\lambda(b_{ij})$ is a Čech coboundary. Thus,
after refining the cover if necessary, there exist local sections
\[
        s_i\in \Gamma(U_i,\mathcal H)
\]
such that
\[
        u_\lambda(b_{ij})=s_j-s_i
        \qquad \text{on } U_{ij}.
\]

By Lemma~\ref{lem:action-of-lift}, the transition from the Higgs bundle
associated with $\widetilde{\mathfrak X}_1$ to that associated with
$\widetilde{\mathfrak X}_2$ is obtained by multiplying the original gluing
maps by
\[
        \exp\bigl(\rho\,\theta(a_{ij})\bigr).
\]
Since $\theta=p\theta'$ and $b_{ij}=\rho a_{ij}$, we have
\[
        \rho\,\theta(a_{ij})
        =
        p\,u_\lambda(b_{ij})
        =
        p(s_j-s_i).
\]
The sections $s_i$ lie in the commutative algebra $\mathcal H$. Therefore
the endomorphisms $ps_i$ commute on overlaps, and the $p$-adic exponential
gives well-defined automorphisms
\[
        \varphi_i:=\exp(ps_i)\in
        \Gamma(U_i,\operatorname{Aut}(\mathcal H,\theta)).
\]
Moreover,
\[
        \exp\bigl(\rho\,\theta(a_{ij})\bigr)
        =
        \exp(p(s_j-s_i))
        =
        \exp(ps_j)\exp(-ps_i)
        =
        \varphi_j\varphi_i^{-1}.
\]
Thus the twisting cocycle appearing in Lemma~\ref{lem:action-of-lift} is a
non-abelian Čech coboundary. Consequently
\[
        \frakS_{\widetilde{\mathfrak X}_2}
        \bigl(\frakS_{\widetilde{\mathfrak X}_1}^{-1}(\mathcal H,\theta)\bigr)
        \simeq
        (\mathcal H,\theta).
\]
Since the two $A_2$-liftings were arbitrary, $(\mathcal H,\theta)$ is
integrally lift-independent.
\end{proof}

\bibliographystyle{alpha}
\bibliography{ref}

@misc{BhattpHTnotes,
  author       = {Bhargav Bhatt},
  title        = {Aspects of $p$-adic {Hodge} theory},
  howpublished = {\url{https://www.math.ias.edu/~bhatt/teaching/mat517f25/pHT-notes.pdf}},
  year         = {Fall 2025},
}

@article{Faltings_2005, title={{A $p$-adic Simpson correspondence}}, volume={198}, ISSN={0001-8708}, url={http://dx.doi.org/10.1016/j.aim.2005.05.026}, DOI={10.1016/j.aim.2005.05.026}, number={2}, journal={Advances in Mathematics}, publisher={Elsevier BV}, author={Faltings, Gerd}, year={2005}, month=dec, pages={847–862} }

@article{MinWang2024,
  author   = {Min, Yu and Wang, Yupeng},
  title    = {{Integral $p$-adic non-abelian Hodge theory for small representations}},
  journal  = {Advances in Mathematics},
  year     = {2024},
  volume   = {458},
  part     = {A},  
  pages    = {109950},
  issn     = {0001-8708},
  doi      = {10.1016/j.aim.2024.109950},
}

@misc{stacks-project,
  author       = {The {Stacks project authors}},
  title        = {The {Stacks} project},
  howpublished = {\url{https://stacks.math.columbia.edu}},
  year         = {2026},
}

@article{heuer2023padic,
  author  = {Heuer, Ben},
  title   = {{A {$p$}-adic Simpson correspondence for smooth proper rigid varieties}},
  journal = {Inventiones mathematicae},
  year    = {2025},
  volume  = {240},
  number  = {1},
  pages   = {261--312},
  month   = apr,
  issn    = {1432-1297},
  doi     = {10.1007/s00222-025-01321-4},
  url     = {https://doi.org/10.1007/s00222-025-01321-4}
}

@article {HeuerXu2026moduli,
    AUTHOR = {Heuer, Ben and Xu, Daxin},
     TITLE = {{$p$}-adic non-abelian {H}odge theory for curves via moduli
              stacks},
   JOURNAL = {J. Amer. Math. Soc.},
  FJOURNAL = {Journal of the American Mathematical Society},
    VOLUME = {39},
      YEAR = {2026},
    NUMBER = {3},
     PAGES = {625--695},
      ISSN = {0894-0347,1088-6834},
   MRCLASS = {14F30 (14D23 14G22 14G45 14H10)},
  MRNUMBER = {5057669},
       DOI = {10.1090/jams/1065},
       URL = {https://doi.org/10.1090/jams/1065},
}

@article{Heuer_HodgeTatespectralsequence,
author = {Heuer, Ben},
title = {The relative {Hodge–Tate} spectral sequence for rigid analytic spaces},
journal = {Journal of the London Mathematical Society},
volume = {112},
number = {4},
pages = {e70318},
doi = {https://doi.org/10.1112/jlms.70318},
url = {https://londmathsoc.onlinelibrary.wiley.com/doi/abs/10.1112/jlms.70318},
eprint = {https://londmathsoc.onlinelibrary.wiley.com/doi/pdf/10.1112/jlms.70318},
year = {2025}
}

@article{anschutz2023smallpadicsimpsoncorrespondence,
  author    = {Ansch{\"u}tz, Johannes and Heuer, Ben and Le Bras, Arthur-C{\'e}sar},
  title     = {{The small $p$-adic {S}impson correspondence in terms of moduli spaces}},
  journal   = {Mathematical Research Letters},
  year      = {2024},
  note = {To appear. arXiv:2312.07554},
  eprint    = {2312.07554},
  archivePrefix = {arXiv},
  primaryClass = {math.AG},
}

@misc{HuSunZuo2025Isomonodromic,
  author        = {Hu, Tianzhi and Sun, Ruiran and Zuo, Kang},
  title         = {{Non-Abelian Kodaira-Spencer Map} and
                   non-existence of holomorphic isomonodromic
                   deformation of {Higgs} bundles over
                   {Teichm{\"u}ller} spaces},
  year          = {2026},
  eprint        = {2512.15478v3},
  archivePrefix = {arXiv},
  primaryClass  = {math.AG},
  doi           = {10.48550/arXiv.2512.15478},
  url           = {https://arxiv.org/abs/2512.15478v3},
  howpublished = {arXiv:2512.15478v3}
}

@article{JuEtAl2026Rethlas,
  title={{Automated Conjecture Resolution with Formal Verification}},
  author={Ju, Haocheng and Gao, Guoxiong and Jiang, Jiedong and Wu, Bin and Sun, Zeming and Liu, Shurui and Chen, Leheng and Wang, Yutong and Wang, Yuefeng and Wang, Zichen and He, Wanyi and others},
  journal={arXiv preprint arXiv:2604.03789},
  pages = {1--54},
  year={2026}
}

@misc{sheng2024smallpadic,
      title={{The small $p$-adic Simpson correspondence in the semi-stable reduction case}}, 
      author={Mao Sheng and Yupeng Wang},
      year={2024},
      eprint={2410.09685},
      archivePrefix={arXiv},
      primaryClass={math.AG},
      url={https://arxiv.org/abs/2410.09685},
      howpublished = {arXiv:2410.09685}
}

@book{abbes2016padic,
  title={{The $p$-adic Simpson Correspondence}},
  author={Abbes, Ahmed and Gros, Michel and Tsuji, Takeshi},
  series={Annals of Mathematics Studies},
  volume={193},
  year={2016},
  publisher={Princeton University Press},
  address={Princeton, NJ}
}

@misc{abbes2026twistinghiggsmodulesfunctorial,
      title={Twisting {Higgs} Modules and Functorial Aspects of the p-adic {Simpson} Correspondence}, 
      author={Ahmed Abbes and Michel Gros and Takeshi Tsuji},
      year={2026},
      eprint={2605.02664},
      archivePrefix={arXiv},
      primaryClass={math.AG},
      url={https://arxiv.org/abs/2605.02664}, 
      howpublished = {arXiv:2605.02664}
}

@article{wang2023padic,
  title={A $p$-adic {Simpson} correspondence for rigid analytic varieties},
  author={Wang, Yupeng},
  journal={Algebra \& Number Theory},
  volume={17},
  number={8},
  pages={1453--1499},
  year={2023},
  publisher={Mathematical Sciences Publishers},
  doi={10.2140/ant.2023.17.1453}
}

@article{Simpson1992,
  author  = {Simpson, Carlos T.},
  title   = {{Higgs bundles and local systems}},
  journal = {Publications Math{\'e}matiques de l'IH{\'E}S},
  volume  = {75},
  year    = {1992},
  pages   = {5--95},
  doi     = {10.1007/BF02699491},
  url     = {https://doi.org/10.1007/BF02699491}
}

@article{LandesmanLitt2024,
  author  = {Landesman, Aaron and Litt, Daniel},
  title   = {{Prill's problem}},
  journal = {Algebraic Geometry},
  year    = {2024},
  volume  = {11},
  number  = {2},
  pages   = {290--295},
  doi     = {10.14231/ag-2024-009},
  url     = {https://algebraicgeometry.nl/article/20170/}
}

@article{Conrad2006ModularCurves,
  author  = {Conrad, Brian},
  title   = {{Modular Curves and Rigid-Analytic Spaces}},
  journal = {Pure and Applied Mathematics Quarterly},
  volume  = {2},
  number  = {1},
  pages   = {29--110},
  year    = {2006},
  url     = {https://www.intlpress.com/site/pub/files/_fulltext/journals/pamq/2006/0002/0001/PAMQ-2006-0002-0001-a003.pdf}
}

@article{Conrad2006RelativeAmpleness,
  author  = {Conrad, Brian},
  title   = {Relative ampleness in rigid geometry},
  journal = {Annales de l'Institut Fourier},
  volume  = {56},
  number  = {4},
  pages   = {1049--1126},
  year    = {2006},
  doi     = {10.5802/aif.2207},
  url     = {https://numdam.org/item/AIF_2006__56_4_1049_0/}
}

@article{Guo2023,
    AUTHOR = {Guo, Haoyang},
     TITLE = {Hodge-{T}ate decomposition for non-smooth spaces},
   JOURNAL = {J. Eur. Math. Soc. (JEMS)},
  FJOURNAL = {Journal of the European Mathematical Society (JEMS)},
    VOLUME = {25},
      YEAR = {2023},
    NUMBER = {4},
     PAGES = {1553--1625},
      ISSN = {1435-9855,1435-9863},
   MRCLASS = {14G22 (14C30 14E15 14G20)},
  MRNUMBER = {4577971},
MRREVIEWER = {Christopher\ David\ Lazda},
       DOI = {10.4171/jems/1222},
       URL = {https://doi.org/10.4171/jems/1222},
}

@book{Hartshorne1977,
  author    = {Hartshorne, Robin},
  title     = {Algebraic Geometry},
  series    = {Graduate Texts in Mathematics},
  volume    = {52},
  publisher = {Springer-Verlag},
  address   = {New York--Heidelberg},
  year      = {1977},
  isbn      = {978-0-387-90244-9},
  doi       = {10.1007/978-1-4757-3849-0},
  url       = {https://doi.org/10.1007/978-1-4757-3849-0}
}

@incollection{Laksov1981,
  author    = {Laksov, Dan},
  title     = {{Weierstrass} points on curves},
  booktitle = {Tableaux de Young et foncteurs de Schur en alg{\`e}bre et g{\'e}om{\'e}trie -- Torun, Pologne, 1980},
  series    = {Ast{\'e}risque},
  number    = {87--88},
  pages     = {221--247},
  year      = {1981},
  publisher = {Soci{\'e}t{\'e} math{\'e}matique de France},
  mrnumber  = {646822},
  url       = {https://www.numdam.org/item/AST_1981__87-88__221_0/}
}

\end{document}